\documentclass[11pt,reqno]{amsart}

\usepackage{amssymb, amsmath, amsthm}
\usepackage{hyperref}
\usepackage{verbatim}
\usepackage{amscd}   
\usepackage[all]{xy} 
\usepackage{tikz-cd}
\usepackage{cases}
\usepackage{array}
\usepackage{calligra,mathrsfs}

\makeatletter
\@namedef{subjclassname@2020}{%
  \textup{2020} Mathematics Subject Classification}
\makeatother

\title{Local cohomology on a subexceptional series of representations}

\author{Andr\'as C. L\H{o}rincz}
\address{Institut f\"ur Mathematik, Humboldt--Universit\"at zu Berlin, Berlin, Germany 12489}
\email{lorincza@hu-berlin.de \vspace{0.1in}}

\author{Jerzy Weyman}
\address{Department of Mathematics, University of Connecticut, Storrs, CT, USA 06269 \newline
\indent Instytut Matematyki, Uniwersytet Jagiello\'nski, Krak\'ow 30-348, Poland}
\email{jerzy.weyman@uconn.edu, jerzy.weyman@owa.uj.edu.pl}


 \newtheorem{theorem}{Theorem}[section]
 \newtheorem{lemma}[theorem]{Lemma}
 \newtheorem{corollary}[theorem]{Corollary}
 \newtheorem{prop}[theorem]{Proposition}
 
 \newtheorem{example}{Example}[section]

\theoremstyle{remark}
 \newtheorem{remark}[theorem]{Remark}
  \newtheorem*{notation}{Notation}

\newcommand{\defi}[1]{{\upshape\sffamily #1}}
\DeclareMathOperator{\ShHom}{\mathscr{H}\text{\kern -3pt {\calligra\large om}}\,}

\newcommand{\C}{\mathbb{C}}
\newcommand{\Z}{\mathbb{Z}}

\newcommand{\bo}{\bigoplus}

\newcommand{\D}{\mathcal{D}}

\newcommand{\bw}{\bigwedge}
\renewcommand{\det}{\textrm{det}}
\newcommand{\lie}{\mathfrak{g}}

\renewcommand{\ll}{\lambda}
\newcommand{\LL}{\Lambda}

\newcommand{\oo}{\otimes}

\newcommand{\F}{\mathcal{F}}
\renewcommand{\L}{\mathcal{L}}
\renewcommand{\O}{\mathcal{O}}

\newcommand{\R}{\mathcal{R}}
\newcommand{\Q}{\mathcal{Q}}
\newcommand{\V}{\mathcal{V}}

\newcommand{\Ann}{\operatorname{Ann}}

\newcommand{\Ext}{\operatorname{Ext}}

\newcommand{\gr}{\operatorname{gr}}
\newcommand{\Hom}{\operatorname{Hom}}

\newcommand{\rep}{\operatorname{rep}}

\newcommand{\SL}{\operatorname{SL}}

\newcommand{\Sp}{\operatorname{Sp}}

\newcommand{\Sym}{\operatorname{Sym}}

\newcommand{\Tot}{\operatorname{Tot}}

\newcommand\charC{{\operatorname{charC}}}

\newcommand{\reg}{\operatorname{reg}}
\newcommand{\codim}{\operatorname{codim}}

\renewcommand{\det}{\operatorname{det}}

\newcommand{\opmod}{\operatorname{mod}}
\newcommand{\op}{\operatorname}
\newcommand{\CM}{\operatorname{reg}}

\newcommand{\bb}[1]{\mathbb{#1}}

\newcommand{\mc}[1]{\mathcal{#1}}
\newcommand{\mf}[1]{\mathfrak{#1}}
\newcommand{\ol}[1]{\overline{#1}}

\def\lra{\longrightarrow}

\numberwithin{equation}{section}

\begin{document}

 \begin{abstract} 
We consider a series of four subexceptional representations coming from the third line of the Freudenthal--Tits magic square; using Bourbaki notation, these are fundamental representations $(G',X)$ corresponding to $(C_3, \omega_3),\, (A_5, \omega_3), \, (D_6, \omega_5)$ and $(E_7, \omega_6)$. In each of these four cases, the group $G=G'\times \C^*$ acts on $X$ with five orbits, and many invariants display a uniform behavior, \emph{e.g.} dimension of orbits, their defining ideals and the character of their coordinate rings as $G$-modules. In this paper, we determine some more subtle invariants and analyze their uniformity within the series. We describe the category of $G$-equivariant coherent $\D_X$-modules as the category of representations of a quiver with relations. We construct explicitly the simple $G$-equivariant $\D_X$-modules and compute the characters of their underlying $G$-structures. We determine the local cohomology groups with supports given by orbit closures, determining their precise $\D_X$-module structure. As a consequence, we calculate the intersection cohomology groups and Lyubeznik numbers of the orbit closures. While our results for the cases $(A_5, \omega_3), \, (D_6, \omega_5)$ and $(E_7, \omega_6)$ are still completely uniform, the case $(C_3, \omega_3)$ displays a surprisingly different behavior. We give two explanations for this phenomenon: one topological, as the middle orbit of $(C_3, \omega_3)$ is not simply connected; one geometric, as the closure of this orbit is not Gorenstein.
\end{abstract}

\maketitle

\section{Introduction}\label{sec:intro}

The subexceptional series coming the third line of the Freudenthal--Tits magic square corresponds to the following four Dynkin formats: $C_3, \, A_5,  \, D_6,  \, E_7 \,$ \cite[Page 168]{freu}. This line stands for the $5$-dimensional symplectic geometries as explained in \cite{freu}. The third line of the extended magic square has the six Dynkin formats \cite[Section~6]{series}:
\begin{equation}\label{eq:series}
A_1, \quad A_1 \times A_1 \times A_1, \quad  C_3, \quad  A_5, \quad  D_6, \quad E_7.
\end{equation}
There are respective parameters
\[m=-2/3, 0, 1, 2, 4, 8.\]
For each of the corresponding Lie groups $G'$, there is a preferred irreducible representation $X$  \cite[Section~6]{series} that displays some uniform behavior within the series. In all (but the first) cases, $G=G' \times \C^*$ acts on $X$ with five orbits, and $\dim X = 6m+8$.

The first, second, and fourth representations from the series (\ref{eq:series}) correspond to the space of binary cubic forms, of $2\times 2\times 2$ hypermatrices and of alternating senary 3-tensors, respectively. The equivariant $\D$-modules and local cohomology modules for these representations are studied in detail in the articles \cite{bindmod}, \cite{mike} and \cite{senary}, respectively. In this paper we complete the analogous study for the rest of the representations within the series (\ref{eq:series}), emphasizing the uniformity of the methods and results. Further, this completes a necessary step toward the classification of such objects on irreducible representations of reductive groups with finitely many orbits that has been initiated through several articles \cite{catdmod,bindmod,claudiu1,claudiu4,lHorincz2018iterated,senary,mike,mike2,lorspher}, see also \cite{kimu} and references therein for their Bernstein--Sato polynomials.

For the Dynkin diagrams $C_3, D_6, E_7$, we use the following conventions on the ordering of nodes:  

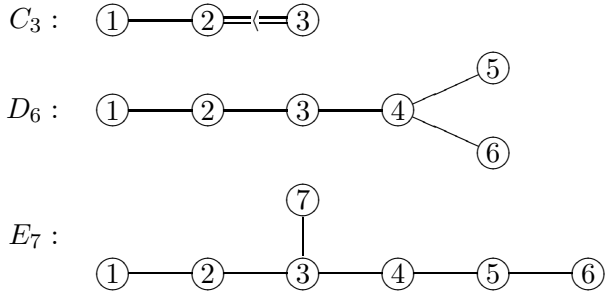
\begin{align}
C_3: & \hspace{0.2in}\xymatrix{ *+[o][F]{1} \ar@{-}[r] &  *+[o][F]{2} &  *+[o][F]{3}\ar@{=}[l] |{\langle}}\nonumber\\
D_6: & \hspace{0.2in}\vcenter{\xymatrix@R-1.8pc{ & & & & *+[o][F]{5}\\ *+[o][F]{1} \ar@{-}[r] &*+[o][F]{2} \ar@{-}[r] &*+[o][F]{3} \ar@{-}[r] &*+[o][F]{4} \ar@{-}[ur] \ar@{-}[dr] & \\ & & & & *+[o][F]{6}}}\nonumber\\
E_7: & \hspace{0.2in}\vcenter{\xymatrix@R-0.8pc{& & *+[o][F]{7} \ar@{-}[d]& & & \\ *+[o][F]{1} \ar@{-}[r]&*+[o][F]{2} \ar@{-}[r]&*+[o][F]{3} \ar@{-}[r]&*+[o][F]{4} \ar@{-}[r]&*+[o][F]{5} \ar@{-}[r]&*+[o][F]{6} }} \nonumber
\end{align}

With the exception of Section \ref{sec:last}, throughout the article the representation $(G',X)$ always denotes either $(C_3, \omega_3),\, (A_5, \omega_3), \, (D_6, \omega_5)$ or $(E_7, \omega_6)$, which come from the third line of the Freudenthal--Tits magic square. Here $(D_6, \omega_5)$ corresponds to the (even) half-spin representation. To display the uniformity of the results better we include $(A_5, \omega_3)$, albeit most results in this case are worked out in \cite{senary}.

Throughout $S=\C[X]$ is the polynomial ring on $X$, and $\D=\D_X$ is the Weyl algebra of differential operators on $X$ with polynomial coefficients. Let $\opmod_G(\D_X)$ denote the category of $G$-equivariant coherent $\D$-modules on $X$ (which are regular and holonomic in our situation). According to the Riemann--Hilbert correspondence, equivariant $\D$-modules correspond to equivariant perverse sheaves, and the simple equivariant $\D$-modules correspond to irreducible equivariant local systems on the orbits. However, their explicit realization is in general a difficult problem (see Open Problem 3 in \cite[Section~6]{mac-vil}).

The representations in the subexceptional series have finitely many orbits, but are not spherical varieties. For the irreducible representations of the latter kind, the categories of equivariant $\D$-modules have been described in \cite{catdmod} and character formulas have been given in \cite{claudiu1,claudiu5, lorspher}. For our four cases the algebra of covariants $S^U$ is still a polynomial ring \cite{brion} (here $U$ denotes a maximal unipotent subgroup of $G$), and they are some of the few representations beyond the spherical case where plethysm formulas for $S^U$ are known. We determine the analogous character formulas for all simple equivariant $\D$-modules, and also provide explicit $\D$-module realizations. The formulas are written as uniformly as possible within the series. 

In the case $(C_3, \omega_3)$ there is an extra simple, since its middle orbit is not simply connected. Moreover, the fact that there are no (semi)-invariant sections for the simple $\D$-module corresponding to the trivial local system on the middle orbit is reflected by the fact that this orbit closure is the only one that is not Gorenstein. In more general setting, we establish a link between the roots of Bernstein--Sato polynomials of semi-invariant polynomials and the Castelnuovo--Mumford regularity of Gorenstein orbit closures (Section \ref{sec:last}).

As the group is acting with finitely many orbits, the category $\opmod_G(\D_X)$ of equivariant $\D$-modules is equivalent to the category of finite-dimensional representations of a quiver with relations (see \cite{vilonen} and \cite{catdmod}). We determine the quiver structure of the category of $\opmod_G(\D_X)$ (see Section \ref{sec:quivmod}). The quivers appear also in \cite{catdmod} and \cite{senary}, and have finitely many indecomposable representations that are described explicitly \cite[Theorem 2.11]{catdmod}. Again, only the case $(C_3, \omega_3)$ displays exceptional behavior as the equivariant $\D$-module corresponding to the trivial local system on the middle orbit is disconnected from the rest.

For any $G$-stable closed subset $Z$ in $X$, the local cohomology modules $H^i_Z(S)$ are $G$-equivariant coherent $\D$-modules, for all $i\geq 0$. While being objects of great interest, explicit computations of local cohomology modules are in general difficult. Several results been obtained for representations with finitely many orbits (see \cite{claudiu2,claudiu3,claudiu4,bindmod,lHorincz2018iterated,mike,mike2}). In spirit of these results, in Section \ref{sec:loccoh} we determine the explicit $\D$-module structure of local cohomology modules for our series. The results are uniform with respect to the parameter $m$, with some discrepancy in the case $(C_3,\omega_3)$ again.

The article is organized as follows. In Section \ref{sec:prelim}, we introduce the basic terminology regarding representations of reductive algebraic groups (Section \ref{sec:rep}), equivariant $\D$-modules (Section \ref{sec:equivd}), representation theory of quivers (Section \ref{sec:quiv}) and the Borel--Weil--Bott theorem (Section \ref{sec:bott}). In Section \ref{sec:cat} we describe explicitly the category equivariant $\D$-modules. We compute the $\D$-module (Section \ref{sec:simples}) and $G$-module (Section \ref{sec:character} and Theorem \ref{thm:D2}) structure of the simple objects explicitly, and we determine the quivers corresponding to the categories (Section \ref{sec:quivmod}). In Section \ref{sec:loccoh} we determine all the 
local cohomology modules with supports given by orbit closures. As an application of our main results, we calculate the intersection cohomology groups and Lyubeznik numbers of orbit closures (Section \ref{sec:lyubint}). In Section \ref{sec:last}, we establish a connection in a more general context between Castelnuovo--Mumford regularity of Gorenstein orbit closures and the roots of the Bernstein--Sato polynomials of semi-invariant polynomials.

\section{Preliminaries}\label{sec:prelim}

Throughout we use the Bourbaki notation for irreducible highest weight representations. The pair  $(G',X)$ denotes either $(C_3, \omega_3),\, (A_5,\omega_3),\, (D_6, \omega_5)$ or $(E_7, \omega_6)$, and we fix the parameter $m=1,2,4,8$ respectively.

\subsection{Representations and Characters} \label{sec:rep}

Let $\LL$ be the set of isomorphism classes of finite dimensional irreducible (rational) representations of the simple, simply connected algebraic group $G'$. We identify $\LL$ with the set of \defi{dominant} integral weights. We put $\omega_i$ for the $i$th fundamental representation of $G'$, and write $(a_1,a_2,\dots, a_k)$ for the highest weight of the irreducible $G'$-module $a_1 \omega_1 + a_2 \omega_2 + \dots + a_k \omega_k$, where $a_i\in \Z_{\geq 0}$ and $k$ is the rank of $G'$. The Cartan product of two irreducible $G'$-modules is $V_\lambda \cdot V_\mu := V_{\lambda+\mu}$, where $\lambda,\mu \in \LL$. 

We deal mostly with $\Z$-\defi{graded} $G'$-modules, that is, representations of the group $G=G' \times \C^*$. We keep track of the grading using a parameter $t$.  A rational $G$-representation (possibly infinite-dimensional) $M$ is \defi{admissible} if each irreducible $G$-representation appears (up to isomorphism) with finite multiplicity in $M$ (i.e. each isotypical component is finite-dimensional). Equivalently, any graded piece $M_d$ ($d\in \Z$) decomposes as a direct sum of $G'$-representations along the isotypical components
$$
M_d=\bo_{\ll \in \LL} V_\ll^{\oplus m^d_{\ll}(M)},
$$
with $m^d_{\ll}(M)\in \Z_{\geq 0}$. Also, we will use the notation 
\[[M]=\bigoplus_{d\in \Z} t^d M_d.\]
For $\ll \in \LL$ we can make sense of admissible $G$-representations (where $1$ denotes the trivial $G'$-module)
\[ \frac{1}{1-t V_\ll} = 1 + t V_{\ll} + t^2 V_{2\ll} + \cdots.\]

\subsection{Equivariant $\D$-modules} \label{sec:equivd}
A $\D_X$-module $M$ is (strongly) \defi{equivariant} if we have a $\D_{G\times X}$-isomorphism $\tau: p^*M \rightarrow m^*M$, where $p: G\times X\to X$ denotes the projection and $m: G\times X\to X$ the map defining the action, with $\tau$ satisfying the usual compatibility conditions (see \cite[Definition 11.5.2]{htt}). 

Another characterization of equivariant $\D$-modules is as follows. Let $\mf{g}$ denote the Lie algebra of $G$. Differentiating the $G$-action on $X$ induces a map from $\mf{g}$ to space of vector fields on $X$, hence a map $\mf{g} \to \D_X$.
Then the $\D$-module $M$ is equivariant if and only if it is endowed with an algebraic $G$-action, such that differentiating this action we recover the $\mf{g}$-action induced from the map $\mf{g} \to \D_X$. 

The category $\opmod_G(\D_X)$ of equivariant $\D$-modules is a full subcategory of the category $\opmod(\D_X)$ of all coherent $\D$-modules, and is closed under taking subquotients. If $Z$ is a $G$-stable closed subset of $X$, we denote by $\opmod_G^Z(\D_X)$ the full subcategory of $\opmod_G(\D_X)$ of equivariant $\D$-modules with support contained in $Z$.

In all our cases $G$ acts on $X$ with finitely many orbits. This implies that every module in $\opmod_G(\D_X)$ is regular and holonomic \cite[Theorem~11.6.1]{htt}. The category $\opmod_G(\D_X)$ is equivalent to the category of finite-dimensional representations of a quiver with relations (see \cite[Theorem 4.3]{vilonen} or  \cite[Theorem 3.4]{catdmod}; for quivers see Section \ref{sec:quiv}). For more details on categories of equivariant $\D$-modules, cf. \cite{catdmod}.

Given an equivariant map between two $G$-varieties, (derived) pushforward and pullback of $\D$-modules preserves equivariance. In particular, so do local cohomology functors $H^i_Z(\bullet)$, for $Z$ an orbit closure in $X$. Namely, for each $i\geq 0$ and each $M\in\opmod_G(\D_X)$, the $i$-th local cohomology module $H^i_Z(M)$ of $M$ with support in $Z$ is an element of $\opmod_G^Z(\D_X)$.

Since $G$ is reductive, another construction of objects in $\opmod_G(\D_X)$ comes from considering the (twisted) Fourier transform \cite[Section 4.3]{catdmod}. This functor gives a self-equivalence
\[\F : \opmod_G(\D_X) \xrightarrow{\sim} \opmod_G(\D_X).\]
For $M\in \opmod_G(\D_X)$ we have as $G$-modules
\begin{equation}\label{eq:fourier}
\F(M) \cong M^* \cdot  \det X^\ast.
\end{equation}
Throughout we work with the convention that polynomials live in non-negative degrees, and note that $X\cong X^\ast$ as $G'$-representations.

For a (rational, finite-dimensional) representation $V$ of $G$, we denote by $\mc{P}(V):= \D_X \oo_{U\lie} V$ the induced $G$-equivariant $\D_X$-module (see \cite[Section 2]{catdmod}), which is projective in $\opmod_G(\D_X$), and
\begin{equation}\label{eq:PVHom}
\Hom_{\D_X} (\mc{P}(V), M) = \Hom_{G}(V, M),
\end{equation}
for any $M \in \opmod_G(\D_X)$.

The action of $G$ on $X$ extends to an action on $S$ and the character of $S$ is given by (see \cite{brion} or \cite[Section 6]{series}):
\begin{equation}\label{eq:charS}
[S]=\frac{1}{(1-tX)(1-t^2 \lie')(1-t^3 X)(1-t^4)(1-t^4 X_4)},
\end{equation}
where $\lie' = 2\omega_1, \, \omega_1+\omega_5, \, \omega_2, \, \omega_1$ is the adjoint representation and $X_4=2 \omega_2, \, \omega_2+\omega_4, \, \omega_4, \, \omega_5$ for $(G',X)=(C_3,\omega_3), \, (A_5, \omega_3), \, (D_6, \omega_5), \, (E_7, \omega_6)$, respectively. 
For the simple $\D$-module $E=\F(S)$ with support equal to the origin, we have by (\ref{eq:fourier})):
\begin{equation}\label{eq:charE}
[E]=\frac{t^{-6m-8}}{(1-t^{-1}X)(1-t^{-2} \lie')(1-t^{-3} X)(1-t^{-4})(1-t^{-4} X_4)}.
\end{equation}

\subsection{Quivers}\label{sec:quiv}
 
We briefly introduce some basic notions on the representation theory of quivers, following \cite{ass-sim-sko}.  A \defi{quiver} $Q$ is an oriented graph, \emph{i.e.} a pair $Q=(Q_0,Q_1)$ formed by a finite set of vertices $Q_0$ and a finite set of arrows $Q_1$. An arrow $a\in Q_1$ has a head $ha$ and a tail $ta$ which are elements of $Q_0$:
\[\xymatrix{
ta \ar[r]^{a} & ha
}\]

A \defi{relation} in $Q$ is a linear combination of paths of length at least two that have the same source and target. We define a \defi{quiver} (with relations) $(Q,I)$ to be a quiver $Q$ together with a finite collection of relations $I$.

A \defi{representation} $M$ of a quiver $(Q,I)$ is a family of (finite-dimensional) vector spaces $\{M_x\,|\, x\in Q_0\}$ together with linear maps $\{M(a) : M_{ta}\to M_{ha}\, | \, a\in Q_1\}$ that satisfy the relations induced by $I$. A morphism $\phi:M\to N$ of two representations $M,N$ of $(Q,I)$ is a set of linear maps  $\phi = \{\phi(x) : M_x \to N_x\,| \,x\in Q_0\}$, such that for each $a\in Q_1$ we have $\phi(ha)\circ M(a)=N(a)\circ \phi(ta)$. The category $\rep(Q,I)$ of finite-dimensional representations of $(Q,I)$ is Artinian, has enough projectives and injectives, and contains only finitely many simple objects,that are in bijection with the vertices. For the projective cover (resp. injective envelope) of the simple corresponding to a vertex $x\in Q_0$, the dimension of its space at $y\in Q_0$ is given by the number of paths from $x$ to $y$ (resp.\ from $y$ to $x$), considered up to the relations in $I$ (see \cite[Section III.2]{ass-sim-sko}).

\subsection{Borel--Weil--Bott theorem}\label{sec:bott}

In this section, we present some special cases of the Borel--Weil--Bott theorem that we use in Section \ref{subsec:ring}. For more details, see \cite{jerzy}, and especially \cite[Chapters 4,5]{bastoneastwood}, as we will use these conventions. 

We denote the integral weights of a simple group $G$ with the corresponding root system (which in this paper will be $C_3$, $A_5$, $D_6$ or $E_7$), by labeling the Dynkin diagram by integers. Dominant weights correspond to labelings by nonnegative integers.
The weight $\rho$ (the half of sum of positive roots) corresponds to labeling all nodes by $1$.
The action of the Weyl group on the integral weights is described in \cite[Recipe 4.1.3]{bastoneastwood} which states the following.

\begin{prop} The simple reflection $\sigma_\alpha$ acts as follows. To compute $\sigma_\alpha (\lambda)$,
 let $c$ be the coefficient of the node of $X_n$ associated to $\alpha$. Add $c$ to the adjacent coefficients, with multiplicity if there is a multiple edge directed towards the adjacent node, and then replace $c$ by $-c$.
\end{prop}

\begin{example} Here we reproduce \cite[Example 4.1.4]{bastoneastwood}. In each case we reflect at the middle node:
\begin{align}
\xymatrix@1{\stackrel{a}{\bullet} \ar@{-}[r] & \stackrel{b}{\bullet} \ar@{-} [r] & \stackrel{c}{\bullet}}   \quad \Longrightarrow \quad &\xymatrix@1{\stackrel{a+b}{\bullet} \ar@{-}[r] & \stackrel{-b}{\bullet} \ar@{-} [r] & \stackrel{b+c}{\bullet}} \nonumber \\
\xymatrix@1{\stackrel{a}{\bullet} \ar@{-}[r] & \stackrel{b}{\bullet} \ar@{=}[r] |{\langle} & \stackrel{c}{\bullet}}   \quad \Longrightarrow \quad &\xymatrix@1{\stackrel{a+b}{\bullet} \ar@{-}[r] & \stackrel{-b}{\bullet} \ar@{=}[r] |{\langle} & \stackrel{b+c}{\bullet}}  \nonumber \\
\xymatrix@1{\stackrel{a}{\bullet} \ar@{-}[r] & \stackrel{b}{\bullet} \ar@{=}[r] |{\rangle} & \stackrel{c}{\bullet}}   \quad \Longrightarrow \quad &\xymatrix@1{\stackrel{a+b}{\bullet} \ar@{-}[r] & \stackrel{-b}{\bullet} \ar@{=}[r] |{\rangle} & \stackrel{2b+c}{\bullet}} \nonumber
\end{align}
\end{example}

Next we define the affine action of the Weyl group on weights via
$$w^.\lambda := w(\lambda+\rho)-\rho.$$

Recall that the weight $\lambda$ is called \defi{singular} if there exists a nontrivial $w\in W$ such that $w^.\lambda=\lambda$.
For a non-singular weight $\lambda$ there exists a unique $w\in W$ such that $w^.\lambda$ is dominant.

Now we are ready to state the Borel--Weil--Bott theorem.
Recall that parabolic subgroups $P$ of $G$ (up to conjugation) correspond to subsets of simple roots. 
Let us fix $G$ and $P$. The weight $\lambda$ is dominant with respect to $P$ if it is dominant when restricted to the Levi factor $L(P)$.
This means that all the labelings of the simple roots that are in $L(P)$ are nonnegative.

The homogeneous vector bundles on $G/P$ correspond to rational $P$-modules. To a $P$-dominant $\lambda$ we associate the homogeneous bundle ${\mathcal V}(\lambda)$
which is irreducible, i.e. the unipotent radical of $P$ acts on it trivially and with the action of $L(P)$ it is the dual of the highest weight $L(P)$-module coresponding to the restriction of $\lambda$.
The Borel--Weil--Bott theorem calculates the sheaf cohomology of such bundles.

\begin{theorem}\label{thm:bott} Let $G$ be a simply connected complex semisimple Lie group, and $P\subset G$ a parabolic subgroup.
Suppose $\lambda$ is an integral weight for $G$ and dominant with respect to $P$. Consider the homogeneous bundle ${\mathcal V}(\lambda)$ on $G/P$.
Then
\begin{enumerate}
\item If $\lambda$ is singular for the affine Weyl group action, then
$$H^r(G/P, {\mathcal V}(\lambda))=0$$
for all $r$.
\item If $\lambda$ is nonsingular for the affine Weyl group action, then as a representation of $G$,
$$H^{l(w)}(G/P, {\mathcal V}(\lambda))=V_{w^.\lambda}^*,$$
where $w\in W$ is the unique element for which $w^.\lambda$ is dominant. All other cohomology vanishes.
\end{enumerate}
\end{theorem}

We finish this section with some examples.

\begin{example} Let us consider the root system of type $C_3$. Calculate cohomology of the bundle corresponding to the weight $3\omega_1-3\omega_3$
This corresponds to the sequence $(3,0,-3)$, After adding $\rho$ we get the weight $(4,1,-2)$. Applying the reflection at the 3-rd vertex we get 
$(4, -3,2)$. Next we apply the reflection on the second vertex to get $(1,3,-1)$ and then after applying the reflection at the third vertex, we get $(1,1,1)$.
Subtracting $\rho$ we get $(0,0,0)$. Hence, the third cohomology of the bundle is a trivial representation, and zero in all other degrees.
\end{example}

\begin{example} Let us consider the root system of type $D_6$. Calculate cohomology of the bundle corresponding to the dual of the $9$-th symmetric power
of the universal quotient bundle $Q_1$ on the corresponding isotropic Grassmannian.
This corresponds to the weight $(-9,0,0,0,0,0)$. Adding $\rho$ we get $(-8,1,1,1,1,1)$. Applying corresponding reflections we get in turn:
$(8,-7,1,1,1,1)$, $(1,7,-6,1,1,1)$, $(1,1,6,-5,1,1)$, $(1,1,1,5,-4,-4)$, $(1,1,1,1,4,-4)$, $(1,1,1,-3,4,4)$, $(1,1,-2,3,1,1)$, $(1,-1,2,1,1,1)$, $(0,1,1,1,1,1)$. This shows that the weight $(-9,0,0,0,0,0)$ is singular, as a zero appears in the algorithm (the corresponding weight is fixed under reflection at the zero node). This shows that all cohomology groups of our bundle are zero.
\end{example}

\begin{example}
Let us consider the system of type $E_7$. Consider the bundle which is the $a$-th multiple of the dual of the 6-th fundamental representation, with $a\ge 10$.
Our weight is $(0,0,0,0,0,-a,0)$ and the above algorithm shows that the $14$-th cohomology of our bundle is the representation with highest weight $(a-10)\omega_6+(a-10)\omega_7$.
All other cohomology groups are zero. 
\end{example}

\section{The category of equivariant coherent $\D$-modules}\label{sec:cat}

As usual, the pair  $(G',X)$ denotes either $(C_3, \omega_3),\, (A_5, \omega_3), \, (D_6, \omega_5)$ or $(E_7, \omega_6)$, and recall that $m=1,2,4,8$, respectively. Whenever possible, we discuss the properties of the action of $G=G'\times \C^*$ on $X$ in a uniform matter within these four cases.

We have $\dim X = 6m+8$, and $G$ acts on $X$ with five orbits. Accordingly, we write $X= \bigcup_{i=0}^4 O_i $ with $\ol{O}_{i-1}\subset \ol{O}_i$ (where $1\leq i \leq 4$), where $O_0=\{0\}$ and $\ol{O}_4=X$. The hypersurface $\ol{O}_3$ defined by the vanishing of a $G'$-invariant polynomial $f$ of degree $4$ (unique, up to scalar). The codimensions of $\ol{O}_2$ and $\ol{O}_1$ are $m+3$ and $3m+4$, and their defining ideals are generated by $X$ (in degree 3) and $\lie'$ (in degree 2)  from (\ref{eq:charS}), respectively. More details can be found in \cite{saki}, \cite[Section 7]{igusa} and \cite[Exercise 7.17]{jerzy}.

For an orbit $O \subset X$, we denote by $O^\vee \subset X$ its projective (Pyasetskii) dual orbit (see \cite[Section 4.3]{catdmod}). By \cite{knopmenzel}, we have
\begin{equation}\label{eq:pya}
 O_i^\vee = O_{4-i}, \mbox{ for } i=0,\dots, 4.
\end{equation}

\subsection{Fundamental groups of orbits}\label{subsec:funds}

Given a $G$-orbit $O \cong G/H$ of $X$, we call the finite group $H/H^0$ the \defi{component group of} $O$ (here $H^0$ stands for the connected component of $H$ containing the identity). If $O$ is simply connected, then its component group is trivial. The orbits $O_1,O_2,O_3$ are also $G'$-orbits. Note that since $G'$ is simply connected, the fundamental groups of $O_1,O_2,O_3$ are isomorphic to their component groups under the action of $G'$. We proceed by determining the fundamental and component groups of all orbits.

\begin{lemma}\label{lem:piO1} The orbit $O_1$ is simply connected.
\end{lemma}

\begin{proof}
Since $O_1$ is the orbit of the highest weight vector, this follows from \cite[Lemma 4.13]{catdmod}. Alternatively, this also follows from the explicit determination of the generic $G'$-stabilizer of $O_1$ \cite[Lemma 16]{igusa}, which is shown to be connected in  \textit{ibid.}, p. 426.
\end{proof}

The following is the first indication that the case $(C_3,\omega_3)$ is somewhat different from the rest of the cases.

\begin{lemma}\label{lem:piO2} The orbit $O_2$ is simply connected, except for $(G',X)=(C_3,\omega_3)$ when the component group equals $\pi_1 (O_2) = \Z/ 2\Z$.
\end{lemma}

\begin{proof}
The claim about the fundamental groups follows from \cite[Lemma 17]{igusa}. Hence, the component group (under the action of $G$) can be either trivial or $\Z/2\Z$. To show that the latter holds, one can either follow through the computations in  \cite[Lemma 17]{igusa}, or use the fact that we have two non-isomorphic simple equivariant $\D$-modules with support $\ol{O}_2$, namely $H^4_{\ol{O}_2}(S)$ and $\D f^{-2}/ \D f^{-1}$ (see Theorem \ref{thm:simples} (b) and Section \ref{sec:loccoh}).
\end{proof}

\begin{lemma}\label{lem:piO3}
The orbit $O_3$ is simply connected.
\end{lemma}

\begin{proof}
This follows from \cite[Lemma 15]{igusa}.
\end{proof}

\begin{lemma}\label{lem:ratO3}
The variety $\ol{O}_3$ is normal with rational singularities.
\end{lemma}

\begin{proof}
This follows by \cite[Theorem 0.4]{saito}, since the polynomial $b_f (s)/(s + 1)$ has no roots $\geq -1$, where $b_f (s)$ stands for the $b$-function of the invariant $f$ (see Section \ref{sec:simples}).
\end{proof}

\begin{lemma}\label{lem:piO4}
The orbit $O_4$ has component group $\Z/4\Z$, and $\pi_1(O_4)\cong \Z$.
\end{lemma}

\begin{proof}
This follows by \cite[Lemma 4.11 and Remark 4.12]{catdmod} and Lemma \ref{lem:ratO3}.
\end{proof}

\subsection{Simple equivariant $\D$-modules}\label{sec:simples}

We determine the filtrations of the equivariant $\D$-modules $S_f$ and $S_f\cdot \sqrt{f}$ using the $b$-function $b_f(s)$ of $f$. From this we obtain the explicit construction of almost all simple equivariant $\D$-modules.

The roots of $b_f(s)$ are (see \cite[Section 12]{kimu}): 
\begin{equation}\label{eq:roots}
-1, \, r_1:= -\frac{m+3}{2}, \, r_2:=-\frac{2m+3}{2}, \, r_3:=-\frac{3m+4}{2}.
\end{equation}

The following is the corresponding the holonomy diagram  (see \cite{kimu} for more details):
\begin{equation}\label{eq:holonomy}
\xymatrix@C+1pc@1{ O_4 \ar@{-}[r]^{s+1} & O_3 \ar@{-}[r]^{s-r_1} & O_2 \ar@{-}[r]^{s-r_2} & O_1 \ar@{-}[r]^{s-r_3} & O_0 }
\end{equation}

\begin{lemma}\label{lem:localbs}
Let $x_{\lie'}$ denote (up to a non-zero constant) the degree $2$ highest weight vector of $\lie'$ in $S$ (see (\ref{eq:charS})), and $\partial_{\lie'}$ the constant differential operator of degree $-2$  of the same highest weight. Then we have
\[\partial_{\lie'} \cdot f^{s+1} = (s+1)(s-r_1) \cdot x_{\lie'} f^s,\]
hence the local $b$-function of $f$ at a point in $O_2$ is $(s+1)(s-r_1)$.
\end{lemma}

\begin{proof}
This follows by \cite[Theorem 2.5]{bub}, see also \cite[Remark 3.6]{senary}. The statement about the local $b$-function also follows from the holonomy diagram (\ref{eq:holonomy}).
\end{proof}

According to \cite[11.6.2]{htt} and Section \ref{subsec:funds}, the simple equivariant $\D$-modules on $X$ are the following.

\begin{notation}
For each $O_i$ (with $1\leq i \leq 3$), we denote the simple $\D$-module corresponding to the trivial local system on $O_i$ by $L_i$. The $\D$-modules $S$ (the coordinate ring) and its Fourier transform $E$ (the injective envelope of the residue field) correspond to the trivial local system on $O_4$ and $O_0$, respectively. Let $L'_4$ and $L'_2$ be the equivariant simple $\D$-modules corresponding to the non-trivial self-dual local systems on $O_4$ and $O_2$, respectively (the latter only for $(G',X)=(C_3,\omega_3)$). The remaining simple equivariant $\D$-modules with full support will be denoted by $L_{4,1}$ and $L_{4,3}$. 
\end{notation}

Given a coherent $\D$-module $M$, we write $\charC(M)$ for its characteristic cycle (see \cite{kashi}), which is a formal linear combination of the irreducible components of its characteristic variety counted with multiplicities -- these are always closures of conormal bundles to orbits when $M$ is equivariant. For such an $M$ and an orbit $O$, write $\op{mult}_{\ol{T^\ast_{O}X}} M$ for the multiplicity of $\ol{T^\ast_{O}X}$ in $\charC(M)$. Clearly, $\charC(S)=[\ol{T^\ast_{O_0} X}]$ and $\charC(E)=[\ol{T^\ast_{O_4}X}]$. The diagram \ref{eq:holonomy} has an edge between two orbits $O,O'$ if $\ol{T^\ast_{O}X}$ and $\ol{T^\ast_{O'}X}$ intersect in codimension 1.

\begin{lemma}\label{lem:charC}
Let $M \in \opmod_{G}(\D_X)$ and $V\subset M$ a finite-dimensional $G$-submodule such that $\D_X \cdot V = M$. Then for any $G$-orbit $O \subset X$, we have $\op{mult}_{\ol{T^\ast_{O}X}} M \leq \dim V$. 
\end{lemma}

\begin{proof}
As there is a surjective map $\mc{P}(V) \to M$ by (\ref{eq:PVHom}), it is enough to prove the statement for $M=\mc{P}(V)$. Let $\mu : T^* X \to \lie^*$ denote the moment map (see \cite[Section 3.1]{catdmod}), and $\mu^{-1}(0)$ the scheme-theoretic fiber. Since $G$ acts on $X$ with finitely many orbits, the irreducible components of $\mu^{-1}(0)$ (as a variety) are precisely the closures of conormal bundles to the orbits. As shown in the proof of \cite[Proposition 3.14]{catdmod}, we have $\op{mult}_{\ol{T^\ast_{O}X}} \mc{P}(V) \leq \, \dim V \cdot \op{mult}_{\ol{T^\ast_{O}X}} \mu^{-1}(0)$. By \cite{kimu}, the variety $\ol{T^\ast_{O_i}X}$ has a dense $G$-orbit, for each $i=0,\dots,4$. Therefore, by \cite[Lemma 3.12]{catdmod} we have $\op{mult}_{\ol{T^\ast_{O}X}}\mu^{-1}(0)=1$.
\end{proof}

\begin{theorem}\label{thm:simples}
We have the following explicit construction of simple equivariant $\D$-modules:
\begin{itemize}
\item[(a)] When $(G',X)$ is $(A_5,\omega_3),\, (D_6, \omega_5)$ or $(E_7,\omega_6)$:
\begin{itemize}
\item[(0)] $E \cong \D f^{r_3}/ \D f^{-1}$.
\item[(1)] $L_1\cong \D f^{r_2}/\D f^{r_1}$ with $\charC(L_1)=[\ol{T^\ast_{O_1}X}]+[\ol{T^\ast_{O_0} X}]$.
\item[(2)] $L_2\cong \D f^{r_1}/\D f^{r_1+1}$ with $\charC(L_2)=[\ol{T^\ast_{O_2}X}]$.
\item[(3)] $L_3\cong \D f^{-1}/S$ with $\charC(L_3)=[\ol{T^\ast_{O_3}X}]+[\ol{T^\ast_{O_2}X}]+[\ol{T^\ast_{O_1} X}]$.
\item[(4)] $L'_4\cong \D f^{r_1+1}$ with $\charC(L'_4)=[\ol{T^\ast_{O_4}X}]+[\ol{T^\ast_{O_3} X}]$, and $L_{4,i}\cong \D f^{i/4}$ (where $i=1,3$) with $\charC(L_{4,i} )=[\ol{T^\ast_{O_4} X}]+[\ol{T^\ast_{O_3}X}]+[\ol{T^\ast_{O_2}X}]+[\ol{T^\ast_{O_1} X}]+[\ol{T^\ast_{O_0} X}]$.
\end{itemize}
\item[(b)] When $(G',X)$ is $(C_3,\omega_3)$:
\begin{itemize}
\item[(0)] $E \cong \D f^{r_3}/ \D f^{r_2}$.
\item[(1)] $L_1\cong \D f^{r_2}/\D f^{r_2+1}$ with $\charC(L_1)=[\ol{T^\ast_{O_1}X}]$.
\item[(2)] $L'_2\cong \D f^{r_1}/\D f^{-1}$ with $\charC(L'_2)=[\ol{T^\ast_{O_2}X}]+[\ol{T^\ast_{O_1}X}]+[\ol{T^\ast_{O_0}X}]$, and $L_2 \cong H^4_{\ol{O}_2}(S)$ with $\charC(L_2)=[\ol{T^\ast_{O_2}X}]$.
\item[(3)] $L_3\cong \D f^{-1}/S$ with $\charC(L_3)=[\ol{T^\ast_{O_3}X}]$.
\item[(4)] $L'_4\cong \D f^{r_2+1}$ with $\charC(L'_4)=[\ol{T^\ast_{O_4}X}]+[\ol{T^\ast_{O_3} X}]$, and $L_{4,i}\cong S_f \cdot f^{i/4}$ (where $i=1,3$) with $\charC(L_{4,i} )=[\ol{T^\ast_{O_4} X}]+[\ol{T^\ast_{O_3}X}]+[\ol{T^\ast_{O_2}X}]+[\ol{T^\ast_{O_1} X}]+[\ol{T^\ast_{O_0} X}]$.
\end{itemize}
\end{itemize}
\end{theorem}

\begin{proof}
Since part (a) follows as  \cite[Theorem 3.5]{senary}, we give a proof for the case $(G',X)=(C_3,\omega_3)$ only, when $r_1=-2$, $r_2=-5/2$ and $r_3=-7/2$.

First, by \cite[Corollary 6.25]{kashi} the equivariant $\D$-modules of full support $\D f^{-3/2}$ and $S_f \cdot f^{i/4}$ ($i=1,3$) are simple. Having disjoint weights of $G$-semi-invariant elements, these $\D$-modules are pairwise non-isomorphic (alternatively, this also follows by the Riemann--Hilbert correspondence). Viewed as $\D$-modules on $X\setminus \ol{O}_3$, tensoring $S_f \cdot f^{i/4}$ by itself four times yields $S_f$, hence $L_{4,i} \cong S_f \cdot f^{i/4}$ and $L_4' \cong \D f^{-3/2}$.

By \cite[Proposition 4.9]{catdmod}, we have the following filtrations in $S_f$ and $S_f \cdot \sqrt{f}$.
\[0\subsetneq S \subsetneq \D f^{-1} \subsetneq \D f^{-2}, \;\;\; 0\subsetneq \D f^{-3/2}\subsetneq \D f^{-5/2} \subsetneq \D f^{-7/2}.\]
Each of the successive quotients of the filtration has a unique simple $\D$-module quotient, hence we get six non-isomorphic equivariant simple $\D$-modules $S, L^{-1}, L^{-2}, \D f^{-3/2}, L^{-5/2} , L^{-7/2}$ (see \cite[Proposition 4.9]{catdmod} and notation therein) respectively, all having $G'$-invariant sections. On the other hand, by Lemma \ref{lem:H4O2} we see that $H^4_{\ol{O}_2}(S)$ (and hence $L_2$, which is always a submodule) has no $G'$-invariant sections. By Section \ref{subsec:funds} this yields all simple equivariant $\D$-modules, and shows also that $L_2\cong H^4_{\ol{O}_2}(S)$ must be simple (note that the argument also implies that the component group of $O_2$ must be indeed $\Z / 2\Z$, as discussed in Lemma \ref{lem:piO2}).

The local cohomology module $H^1_{\ol{O}_3}(S)=S_f/S$ contains a unique simple $\D$-module, which is $L_3$. The module $\D f^{-1}/S$ is a submodule of $H^1_{\ol{O}_3}(S)$, hence has unique simple sub- and quotient modules $L_3$ and $L^{-1}$, respectively. From $S \subsetneq \D f^{-1} \subsetneq \D f^{-2}$ we see that $f^{-1}$ is the only $G'$-invariant section of $\D f^{-1}/S$ (up to scaling), therefore the only other simple $\D$-module besides $L^{-1}$ that could appear as its composition factor is $L_2$. Hence, $L_3 \cong \D f^{-1}/S \cong L^{-1}$.

Since $E=\F(S)$ has a $G'$-invariant section of degree $-14$, we must have $L^{-7/2} \cong E$ by (\ref{eq:charE}). Similarly, $\F(L_3)\cong L^{-5/2}$, $\F(L_2) \cong L_2$, $\F(L^{-2}) \cong L'_4$ and $\F(L_{4,i})\cong L_{4,i}$ (with $i=1,3$). In particular, $\ol{T^\ast_{O_0} X}$ is an irreducible component of the characteristic variety of $L^{-2}, L_{4,1}$ and $L_{4,3}$ (see \cite[Section 4.3]{catdmod}).

By Lemma \ref{lem:charC}, any equivariant $\D$-module that is generated by a $G$-semi-invariant section has a multiplicity-free characteristic cycle (for comparison, see end of proof of \cite[Theorem 3.5]{senary}). Moreover, for any equivariant indecomposable $\D$-module, its characteristic variety should have irreducible components connected via the holonomy diagram \ref{eq:holonomy} (see \cite[Theorem 6.7]{mac-vil}). We have seen that both $\ol{T^\ast_{O_0} X}$ and $\ol{T^\ast_{O_4} X}$ are components of $\charC(S_f\cdot f^{i/4})$, for $i=0,1,2,3$. In conclusion, we have that $\charC(S_f \cdot f^{i/4})=[\ol{T^\ast_{O_4} X}]+[\ol{T^\ast_{O_3}X}]+[\ol{T^\ast_{O_2}X}]+[\ol{T^\ast_{O_1} X}]+[\ol{T^\ast_{O_0} X}]$, for all $i=0,1,2,3$.

Since $\ol{T^\ast_{O_0} X}$ is a component of the characteristic variety of $L^{-2}$ and its support cannot be $O_0$,  $\ol{T^\ast_{O_1} X}$ must also be a component. Since $\ol{T^\ast_{O_3} X}$ is in $\charC(L_3)$, by (\ref{eq:pya}) $\ol{T^\ast_{O_1} X}$ is a component of the characteristic variety of $\F(L_3) \cong L^{-5/2}$. From the $\D$-modules $L^{-2},L^{-5/2}$, one must be $L_1$ and the other $L'_2$: if $L^{-2}\cong L'_2$, then $\ol{T^\ast_{O_2} X}$ is a component in the characteristic varieties of $L^{-2}$ and $\F(L^{-2})\cong L_4'$; on the other hand, if $L^{-5/2}\cong L'_2$, then  $\ol{T^\ast_{O_2} X}$ is a component in the characteristic varieties of $L^{-5/2}$ and $\F(L^{-5/2})\cong L_3$. Either way, since $[\ol{T^\ast_{O_2}X}]$ appears in $\charC(S_f \cdot f^{i/2})$ with multiplicity one (for $i=0,1$), this shows that $L_2$ cannot be a composition factor of  $S_f \cdot f^{i/2}$. Hence, $S, L_3, L^{-2}$ (resp. $L_4', L^{-5/2} , E$) are all the simples appearing as composition factors in $S_f \cdot f$ (resp. $S_f \cdot \sqrt{f}$), and all with multiplicity one. Therefore, we have $L^{-2} \cong \D f^{-2} / \D f^{-1}$ and $L^{-5/2}\cong \D f^{-5/2} / \D f^{-3/2}$.

We are left to show that we have in fact $L_1 \cong \D f^{-5/2}/ \D f^{-3/2}$ (and hence $L_2' \cong \D f^{-2}/ \D f^{-1}$), for which it is enough to see that the support of $\D f^{-5/2}/ \D f^{-3/2}$ is contained in $\ol{O}_1$. Since $\lie'$ (in degree $2$) from (\ref{eq:charS}) generates the defining ideal of $\ol{O}_1$, we conclude by Lemma \ref{lem:localbs}.
\end{proof}

\subsection{The quiver of $\opmod_{G}(\D_X)$} \label{sec:quivmod}

Here we describe the quivers of $\opmod_{G}(\D_X)$ as discussed in Section \ref{sec:prelim}. The vertices of the quivers are labeled with the simple equivariant $\D$-modules that they correspond to.

\begin{theorem}\label{thm:quiv}
There is an equivalence of categories
$$
\opmod_G(\D_X)\cong \rep(Q,I),
$$
where $\rep(Q,I)$ is the category of finite-dimensional representations of a quiver $Q$ with relations $I$. The vertices $L_{4,i}$ (with $i=1,3$) are isolated, while the rest of quiver $Q$ is given as follows (with the relations $I$ given by all $2$-cycles):
\begin{itemize}
\item[(a)] When $(G',X)$ is $(A_5,\omega_3),\, (D_6, \omega_5)$ or $(E_7,\omega_6)$:
\[\xymatrix@1{S\ar@<0.5ex>[r] & \ar@<0.5ex>[l] L_3 \ar@<0.5ex>[r] & \ar@<0.5ex>[l] E} \quad \quad \quad \xymatrix@1{L_4'\ar@<0.5ex>[r] &\ar@<0.5ex>[l] L_2\ar@<0.5ex>[r] & \ar@<0.5ex>[l] L_1}\]
\item[(b)] When $(G',X)$ is $(C_3,\omega_3)$:
\[\xymatrix@1{S\ar@<0.5ex>[r] & \ar@<0.5ex>[l] L_3 \ar@<0.5ex>[r] & \ar@<0.5ex>[l] L_2'} \quad \quad \quad \xymatrix@1{L_4'\ar@<0.5ex>[r] &\ar@<0.5ex>[l] L_1\ar@<0.5ex>[r] & \ar@<0.5ex>[l] E}\]
and the vertex $L_2$ is isolated.
\end{itemize}

\end{theorem}

\begin{proof}
Again, we sketch the proof for case (b) only, otherwise the result follows as in \cite[Section 3]{senary}. The modules $S_f \cdot f^{i/4}$ are injective hulls in $\opmod_{G}(\D_X)$ \cite[Lemma 2.4]{bindmod}. Using holonomic duality, this shows that the vertices $L_{4,i}$ (with $i=1,3$) are isolated. The module $L_2\cong H^4_{\ol{O}_2}(S)$ being an injective object in $\opmod_G^{\ol{O}_2}(\D_X)$ \cite[Lemma 3.11]{catdmod}, we see that it must be an isolated vertex by taking Fourier transforms and holonomic duality.

The filtration of $S_f \cdot f$ (resp. $S_f \cdot \sqrt{f}$) gives non-split extensions by \cite[Lemma 2.8]{catdmod}. This, arguing by holonomic duality and Fourier transform, yields all the required arrows of the quiver. The fact that there are no arrows between $L_3$ and $L_1$ follows since $L_1$ is not a composition factor for $H^1_f(S)$, which is the injective hull of $L_3$ in the category $\opmod_G^{\ol{O}_3}(\D_X)$ \cite[Lemma 3.11]{catdmod}. Hence, there are no other arrows in the quiver. Due to the structure of their injective hulls, the $2$-cycles at $S$ and $L_4'$ are zero, which in turn implies that the $2$-cycles at $L_2'$ and $E$ are zero as well, by applying Fourier transforms.

We are left to show that the 2-cycles at $L_3$ and $L_1$ are zero. By Theorem \ref{thm:simples}, $L_1$ and $L_3$ are generated by $G$-semi-invariant sections as $\D_X$-modules. Projective objects in $\opmod_G(\D_X)$ that map to these simples and are generated by $G$-semi-invariant sections can be constructed via (\ref{eq:PVHom}). By Lemma \ref{lem:charC}, such projectives must have multiplicity-free characteristic cycles. In particular, $L_3$ and $L_1$ appear only once in the composition series of their respective projective covers, which implies that the 2-cycles at $L_3$ and $L_1$ are zero. 
\end{proof}

While it fixes all isolated vertices, the Fourier transform also behaves differently in the two cases above, as seen in the proof of Theorem \ref{thm:simples}: in (a) it reflects each individual component of the quiver, while in (b) it reflects one component into the other.

Note that (each connected component of) this quiver appears also in  \cite{senary,catdmod}, and $(Q,I)$ has finitely many (isomorphism classes of) indecomposable representations that can be described explicitly \cite[Theorem 2.11]{catdmod}.

\subsection{Characters of equivariant $\D$-modules}
\label{sec:character}

Since $G$ acts on $X$ with finitely many orbits, any equivariant coherent $\D$-module is admissible as a $G$-representation by \cite[Proposition 3.14]{catdmod}. In this section, we describe explicitly the $G$-module structure of all the simple equivariant $\D$-modules. The techniques we use are based on the those in \cite[Section 4]{senary}. 

The characters of $S$ and $E$ are given by (\ref{eq:charS}) and (\ref{eq:charE}), respectively, and that of $S_f$ by (see \cite[(2.6)]{bindmod} for more details on the notation)
\begin{equation}\label{eq:charSf}
[S_f] = \lim_{n\to \infty} [f^{-n}\cdot S] =\frac{1}{(1-tX)(1-t^2 \lie')(1-t^3 X)(1-t^4 X_4)} \cdot t^{4 \Z},
\end{equation}
where $t^{4 \Z}=\sum_{i\in \Z} t^{4i}$. Clearly, $[S_f \cdot f^{i/4}]= [S_f] \cdot t^{i}$, for $i=0,1,2,3$. In particular, we get formulas for $[L_{4,i}]$, for $i=1,3$.

In this section we compute the character $[\D f^{r_1+1}]$ in a uniform matter, and explain how this can be used to compute the characters of all the other $G$-equivariant simple $\D$-modules, with the exception of $L'_2$ in the case $(G',X)=(C_3, \omega_3)$ which will be considered separately in Section \ref{sec:loccoh}. 

First, consider the case when $(G',X)$ is $(A_5,\omega_3),\, (D_6, \omega_5)$ or $(E_7,\omega_6)$. We readily obtain the character of $L_3$ from $[L_3]=[S_f]-[S]-[E]$. If we know $[\D f^{r_1+1}] = [L_4']$, from $\F(L_1) \cong L_4'$ we have by (\ref{eq:fourier}) the relation $[L_1]=[L_4']^\ast\cdot t^{-6m-8}$, and also get $[L_2]=[S_f  \cdot \sqrt{f}]-[L'_4]-[L_1]$.

Similarly when $(G',X)=(C_3,\omega_3)$, given $[\D f^{r_1+1}]=[\D f^{-1}]$ yields $[L_3]=[\D f^{-1}] - [S]$ and $[L_2']=[S_f] - [\D f^{-1}]$. By Fourier transform, we get using (\ref{eq:fourier}) also $[L_1]$ and $[L_4']$ from $\F(L_3) \cong L_1$ and $\F(L_2') \cong L_4'$.

Recall the notation in (\ref{eq:charS}). We have the following formula.

\begin{theorem}\label{thm:charDfr1}
The $G$-character of $D f^{r_1+1}$ is given by
\[[D f^{r_1+1}] =\frac{t^{4(r_1+1)}}{(1-t^{-1} X)(1-X_4)(1-t X)(1-t^2 \lie')(1-t^4)}.\]
\end{theorem}

\begin{proof}
As this follows essentially the same way as \cite[Theorem 4.1]{senary}, we only sketch the proof to display its uniformity and refer to \textit{loc. cit.} for additional details. Put $p=-r_1-1$. For $k\in \bb{Z}_{\geq 0}$, let $I_k$ (resp. $\tilde{I}_k$) be the following $G$-equivariant ideal of $S$ (resp. of $S/(f)$):
\[ I_k = \{ s \in S \mid s \cdot f^{-p-k} \in D_X f^{-p}\}\, \mbox{ and } \, \tilde{I}_k = (I_k+(f))/(f).\]
Let $S^* \subset \D_X$ denote the polynomial ring in the partial derivatives. Denote by $\partial_1 \in S^*$ the operator of degree $-1$ of the same highest weight as $X$, and by $V$ the (unique) representation of $S^*$ isomorphic to $\lie'$ in degree $-2$ (see (\ref{eq:charS})).  By Lemma \ref{lem:localbs}, the ideal $(V)$ generated by $V$ annihilates $f^{-p}$. We have as a $G$-module (see \cite[Chapter 5, Exercise 8]{jerzy})
\[[S^*/(V)] = \frac{1}{1-t^{-1}X}.\]
This shows that the ideal $I_k$ is generated by the elements in the irreducible representation with highest weight vector $(\partial_1^k \cdot f^{-p})\cdot f^{p+k}$, of weight $k \cdot X$ in degree $3k$. As this vector is not in $(f)$, the ideal $\tilde{I}_k$ of $S/(f)$ is generated by the elements in the (unique) representation of highest weight $k\cdot X$ in degree $3k$. We deduce that the highest weights in $\tilde{I}_k$ are of the form
\[t^{a+2b+3c+4d} \cdot\left( (a+c)X+b \cdot \lie' + d \cdot X_4 \right), \,\, \mbox{ where } a,b \geq 0,\, c+d \geq k.\] 
As in the proof of \cite[Theorem 4.1]{senary}, we can readily rewrite this in the desired formula using \cite[Lemma 4.2]{senary}.
\end{proof}


We are left to determine the character of the $\D$-module $L_2\cong H^4_{\ol{O}_4}(S)$ in the case when $(G',X)=(C_3,\omega_3)$. The formula (see Theorem \ref{thm:D2}) is postponed until Section \ref{sec:C3O2} as we calculate it using a different approach.

\section{Local cohomology}
\label{sec:loccoh}

In this section, we determine all the local cohomology modules of the coordinate ring supported in the orbit closures. The strategies used are similar to the ones in \cite[Section 5]{senary}.

\subsection{Local cohomology of $S$}\label{subsec:ring}

The goal in this section is to prove the following theorem. 

\begin{theorem}\label{thm:locring}
The following are all the non-zero local cohomology modules of $S$ with support in an orbit closure:
\begin{itemize}
\item[(a)] When $(G',X)$ is $(A_5,\omega_3),\, (D_6, \omega_5)$ or $(E_7,\omega_6)$:
\begin{itemize}
\item[(0)] $H^{6m+8}_{O_0}(S) = E$. \smallskip
\item[(1)] $H^{3m+4}_{\ol{O}_1}(S) = L_1, \;\; H^{4m+5}_{\ol{O}_1}(S) = E, \;\; H^{5m+5}_{\ol{O}_1}(S) = E$.\smallskip
\item[(2)] $0\to L_2 \to H^{m+3}_{\ol{O}_2}(S) \to L_1 \to 0, \; \; H^{2m+3}_{\ol{O}_2}(S) = L_1, \; \; H^{3m+4}_{\ol{O}_2}(S) = E.$\smallskip
\item[(3)] $0 \to L_3 \to H^1_{\ol{O}_3} (S) \to E \to 0.$
\end{itemize}
\item[(b)] When $(G',X)$ is $(C_3,\omega_3)$:
\begin{itemize}
\item[(0)] $H^{6m+8}_{O_0}(S) = E$. \smallskip
\item[(1)] $0\to L_1 \to H^{3m+4}_{\ol{O}_1}(S) \to E \to 0$.\smallskip
\item[(2)] $H^{m+3}_{\ol{O}_2}(S)=L_2, \;\; H^{3m+4}_{\ol{O}_2}(S)=E.$ \smallskip
\item[(3)] $0 \to L_3 \to H^1_{\ol{O}_3} (S) \to L_2' \to 0.$
\end{itemize}
\end{itemize}
\end{theorem}

Part (3) follows by Theorem \ref{thm:simples} since $H^1_{\ol{O}_3}(S)=S_f/S$. The non-trivial parts are (1) and (2). 

Since $H^{m+3}_{\ol{O}_2}(S)$ (resp. $H^{3m+4}_{\ol{O}_1}(S)$) is the injective envelope of $L_2$ in $\opmod_G^{\ol{O}_2}(\D_X)$ (resp. of $L_1$ in $\opmod_G^{\ol{O}_1}(\D_X)$) by \cite[Lemma 3.11]{catdmod}, the claim about their structures follows by our description of the quiver of $\opmod_G(\D_X)$ in Theorem \ref{thm:quiv}.  In fact, in case (a) (resp. case (b)) we have an isomorphism $H^{m+3}_{\ol{O}_2}(S)\cong S_f \sqrt{f}/\D f^{r_1+1}$ (resp.  $H^{3m+4}_{\ol{O}_1}(S)\cong S_f \sqrt{f}/\D f^{r_2+1}$). 

We proceed with part (1).  Since here we consider local cohomology supported in the cone over a smooth projective variety, there are several results available in this direction \cite{Ogus},\cite{garsab},\cite{switala2015lyubeznik},\cite{LSW},\cite{HP}.

\begin{prop}\label{prop:locO1}
Apart from $H^{3m+4}_{\ol{O}_1}(S)$, the only non-zero local cohomology modules of $S$ with support in $\ol{O}_1$ are $H^{4m+5}_{\ol{O}_1}(S) = H^{5m+5}_{\ol{O}_1}(S) = E$ when $(G',X)$ is $(A_5,\omega_3),\, (D_6, \omega_5)$ or $(E_7,\omega_6)$.
\end{prop}

\begin{proof}
We only need to determine how many copies of $E$ appear in $H^j_{\ol{O}_1}(S)$ for $j>3m+4$, as the latter is always supported on the origin. Note that the highest weight orbit $\ol{O}_1$ is the affine cone over some (partial) flag variety $G/P$. Hence, it is enough to determine the Betti numbers of $G/P$, according to \cite[Main Theorem 1.2]{switala2015lyubeznik} (see also \cite[Theorem]{garsab} and \cite[Theorem 3.1]{LSW}). These numbers are given by the Bruhat decomposition. The Poincar\'e polynomials encoding them are known, and can be computed case-by-case as explained in \cite[Sections 1.11 and 3.15]{hump} using factorization methods.

Let us consider first the case $(G',X)=(C_3,\omega_3)$. By \cite[Section 1.11]{hump} (see also the remark in  \textit{ibid.}, p. 74), the Poincar\'e polynomial of $G/P$ can be written as the quotient of the Poincar\'e polynomials of $G' = \Sp_3$ and $\SL_3$. The respective degrees of fundamental invariants under the action of their Weyl groups are $2,4,6$ and $2,3$. Therefore, by \cite[Theorem 3.15]{hump} their respective Poincar\'e polynomials are
\[\dfrac{(1-q^4)(1-q^8)(1-q^{12})}{(1-q^2)^3}, \quad \mbox{and} \qquad \dfrac{(1-q^4)(1-q^6)}{(1-q^2)^2}.\]
Hence, the Poincar\'e polynomial of $G/P$ is
\begin{equation}\label{eq:poin}
 P_1(q) = (1+q^6) \cdot \frac{1-q^8}{1-q^2}.
\end{equation}

Now let $(G',X)$ be $(A_5,\omega_3),\, (D_6, \omega_5)$ or $(E_7,\omega_6)$, when $m=2,4,8$, respectively. By a similar argument, we obtain that Poincar\'e polynomial of $G/P$ is:
$$
P_m(q)=(1+q^{m+2})(1+q^{22-32/m})\cdot \frac{1-q^{3m+4}}{1-q^2}.
$$
\end{proof}

We are left with proving both parts (2) of Theorem \ref{thm:locring}, which we devote the rest of the section to.

The variety $\ol{O}_3$ is the projective dual of the highest weight orbit, hence given by a discriminant in the sense of \cite{jerzydisc}. Thus, it has a desingularization as the total space $Z=\Tot(\eta^*)$ of a bundle $\eta$ of 1-jets on $G/P$, as described in \cite[Section 1]{jerzydisc}. The space $Z$ is a subbundle of the trivial bundle $G/P \times X$, and we denote the first and the second projection (which yields the desingularization of $\ol{O}_3$) by 
\[p_1:Z \lra G/P, \;\;\;\;\; p_2 : Z \lra \ol{O}_3.\]
We denote by $\xi$ the locally free sheaf on $G/P$ corresponding to the quotient bundle obtained from the inclusion $Z\subset G/P \times X$. Hence, we have the following exact sequence of locally free sheaves on $G/P$:
\[ 0\to \xi \to X\oo \O_{G/P} \to \eta \to 0.\] 

We give the following uniform description of the bundles $\eta$ and $\mu$ following \cite[Section 1]{jerzydisc} (see also \cite[Section 9.3]{jerzy}). The group $P$ is a maximal parabolic which can be represented by distinguishing the corresponding node in the Dynkin diagram. We have $\eta=\eta'(-1)=\eta'\oo \O(1)$ (here $\O(1)\cong \V(X)$ is the twisting sheaf), where $\eta'$ fits in a sequence
\begin{equation}\label{eq:eta}
0\rightarrow \Omega_{G/P}\rightarrow \eta'\rightarrow \O_{G/P}\rightarrow 0.
\end{equation}

The cotangent bundle $\Omega_{G/P}$ in terms of $P$-dominant weights can be described by labeling the distinguished node by $-2$ and the rest of the nodes with the number of edges connecting them to the distinguished node:

\begin{itemize}
\item $(C_3,\omega_3)$: \, $\V(0,2,-2) = S_2 \R$, where the latter is the second symmetric power of the 3-dimensional isotropic tautological subbundle $\R$ (see \cite[Chapter 4, Exercise 9]{jerzy});

\item $(A_5,\omega_3)$: \, $\V(0,1,-2,1,0) = \R \oo Q^*$, where $\R$ (resp. $\Q$) is the tautological subbundle (resp. quotient bundle)  (see \cite[Proposition 3.3.5]{jerzy});

\item $(D_6,\omega_5)$: \, $\V(0,0,0,1,-2,0) = \bw^2 \R$, where $\R$ denotes the 6-dimensional tautological subbundle (see \cite[Chapter 4, Exercise 10]{jerzy});
\item $(E_7,\omega_6)$: \, $\V(0,0,0,0,1,-2,0)$, which is induced from the 27-dimensional representation $V_{\omega_1}$ of $E_6$.
\end{itemize}

Therefore, we can write uniformly $\Omega_{G/P} = \V(X_4-2\cdot X)$.

\begin{lemma}\label{lem:decomeg}
For any $d\geq 0$, we have a decomposition of bundles
\[\Sym_d  \Omega_{G/P} = \bigoplus_{\substack{a,b\geq 0, \,\, c\geq a+b \\ 3c-2b-a=d}} \V(a\cdot \lie'+b\cdot X_4-2c\cdot X).\]
\end{lemma}

\begin{proof}
The case $(C_3,\omega_3)$ follows by \cite[Proposition 2.3.8]{jerzy} (since $\V(0,2,-2) = S_2 \R$). For the case $(A_5,\omega_3)$, see \cite[Section 5.1]{senary}. For $(D_6, \omega_5)$, this follows by \cite[Proposition 2.3.8]{jerzy} (since $\V(0,0,0,1,-2,0) = \bw^2 \R$). For $(E_7,\omega_6)$, see \cite[Section 4]{johnson}.
\end{proof}

The next three results are the extensions of \cite[Lemma 5.2, Propositions 5.3 and 5.4]{senary}.

\begin{lemma}\label{lem:nonsplit}
We have $H^0(G/P, \,\Sym \eta)=S/(f)$, and $H^i(G/P,\, \Sym \eta) = 0$ for $i\geq 0$. Moreover, $\eta'=\eta\oo \O(-1)$ is characterized as the unique nonsplit extension (up to isomorphism) in the sequence (\ref{eq:eta}).
\end{lemma}

\begin{proof}
As the proof is analogous to \cite[Lemma 5.2]{senary}, we only sketch the argument. The first claim follows by Lemma \ref{lem:ratO3}. The fact that the sequence  (\ref{eq:eta}) is non-split follows as in \textit{loc. cit.} from $H^1(G/P, \,\Sym_2 \eta)=0$ and $H^1(G/P, \, \Sym_2 (\Omega_{G/P} \oo \O(1) ) \neq 0$ (the latter follows by Theorem \ref{thm:bott} and Lemma \ref{lem:decomeg}). By Theorem \ref{thm:bott}
\[\Ext^1(\O_{G/P}, \,\Omega_{G/P}) = H^1(G/P, \, \Omega_{G/P}) = \C,\]
from which uniqueness follows.
\end{proof}

Denote the open $U=p_2^{-1}(O_3) \subset Z$. Since $p_2$ is $G$-equivariant and birational, we have $U\cong O_3$ as $G$-varieties.

\begin{prop}\label{prop:divisor}
We have $Z\setminus U = D$, where $D$ is a $G$-stable divisor on $Z$. Moreover, the ideal sheaf of $D$ is $p_1^* (\L)$, where $\O(1) \cong \L \subset \Sym_3 \eta$.
\end{prop}

\begin{proof}
Using Lemma \ref{lem:decomeg}, we see that $\L = \O(1) \subset \Sym_3 (\Omega_{G/P}(1)) \subset \Sym_3 \eta$. By adjunction, we get an inclusion $p_1^*(\L) \subset \O_Z$, which gives a $G$-stable divisor $D \subset X$. As in the proof of \cite[Proposition 5.3]{senary}, we show that all closed $G$-stable proper subsets of $Z$ are contained in $D$. We give a uniform argument as follows. 

Consider the Levi decomposition $P\cong L \ltimes U(P)$. Write $T^*(G/P) = G\times_P V$, for a $P$-representation $V$. 
We write $Z\cong G\times_P N$, with $N$ fitting into the exact sequence of $P$-modules (\ref{eq:eta})
\begin{equation}\label{eq:pee}
0 \to \C \to N \to V \to 0,
\end{equation}
where $C^*$ acts on each space by scaling by the third power. There is a one-to-one correspondence between the $G$-stable closed subsets in $Z$ and the $P$-stable closed subsets in $N$, and $D$ corresponds to some $P$-stable divisor $D'\subset N$. We are left to show that all $P$-stable closed proper subsets in $N$ are contained in $D'$. 

As $\L \subset \Sym_3 (\Omega_{G/P} \oo \O(1))$, there is a corresponding $L$-semi-invariant polynomial $h$ on $V$ of degree 3, and $D'$ is given as the zero-set of $N \to V \xrightarrow{h} \C$. As $L$-modules, we have $N\cong \C \oplus V$. Write $L=L' \times \C^*$ according to $G=G' \times \C^*$. As the $L'$-module $V$ can be obtained using a construction due to Vinberg through a grading of $\lie'$ corresponding to the distinguished node (the node that we associated to $\Omega_{G/P}$), it has finitely many $L'$-orbits \cite{vinberg} (in fact, due to the multiplicity-free decomposition in Lemma \ref{lem:decomeg}, $V$ is a spherical representation, i.e. already a Borel subgroup of $L'$ acts with finitely many orbits; alternatively, see the classification \cite[Section 8]{saki}). Denote the orbits by $O'_1, \dots, O'_n$, with $\ol{O'_n}= V$. From the classification in \cite[Section 7]{saki}, we see that the prehomogeneous space $(L',V)$ is regular, i.e. we have $X \setminus O_n =  h^{-1}(0)$. Hence, $N$ has $2n$ orbits under the action of $L$, given by $(0,O'_i)$ and $(\C^*, O'_i)$, $i=1,\dots, n$, and the only orbit closures not contained in $D'$ are $\ol{(\C^*,O'_n)} = N$ and $\ol{(0,O'_n)}$. But the latter is not $P$-stable, otherwise the sequence (\ref{eq:pee}) would split, and thus contradict Lemma \ref{lem:nonsplit}.
\end{proof}

The inclusion $p_1^*(\L) \subset \O_Z$ of line bundles gives a sequence of inclusions $\O_Z \subset p_1^*(\L^{-1}) \subset p_1^*(\L^{-2}) \subset \dots \, $. Applying $p_{1*}$ we get inclusions $\Sym \eta \, \subset \, \L^{-1} \oo \Sym \eta \, \subset \, \L^{-2} \oo \Sym \eta \, \subset \dots $. The next result follows essentially in the same way as \cite[Proposition 5.4]{senary}.

\begin{prop}\label{prop:localopen}
For each $i\geq 2$, we have an isomorphism of $G$-modules
\[H_{\ol{O}_2}^i(S/(f)) \cong \varinjlim_k H^{i-1} (G/P, \, \L^{-k}\oo\Sym \eta).\]
\end{prop}

We estimate cohomology by working first with the associated graded of $\eta$:
\begin{equation}\label{eq:greta}
\gr \eta= (\Omega_{G/P}\oplus \O_{G/P})\oo \O(1).
\end{equation}

We have the following (multiplicity-free) decomposition in degree $d-3k$ by Lemma \ref{lem:decomeg}:
\begin{equation}\label{eq:decogr}
\L^{-k}\oo\Sym_d (\gr \eta) = \bigoplus_{\substack{a,b\geq 0, \,\, c\geq a+b \\ 3c-2b-a\leq d}} \V(a \cdot \lie'+b\cdot X_4+(d-2c-k)\cdot X).
\end{equation}

Now we analyze the limit maps in Proposition \ref{prop:localopen} above.

\begin{lemma}\label{lem:lim}
For any representation $\ll$ of $G$, the map 
\[H^i(G/P, \,\L^{-k}\oo \Sym \eta) \to H^i(G/P,\, \L^{-k-1} \oo \Sym \eta)\]
is an isomorphism on the level of $\ll$-isotypical components for $k \gg 0$.
\end{lemma}

\begin{proof}
By the Borel--Weil--Bott theorem \ref{thm:bott}, there exists only finitely (isomorphism classes) of irreducible $G$-bundles on $G/P$ that yield the representation $\ll$ in some cohomological degree. The multiplicity for each of these bundles stabilizes in $\L^{-k}\oo\Sym (\gr \eta)$ for $k\gg 0$, since by (\ref{eq:decogr})  $\L^{-k}\oo\Sym_d (\gr \eta)$ has a multiplicity-free decomposition into $G$-bundles (in fact, admissibility suffices). Write
\begin{equation}\label{eq:bundlim}
0 \to \L^{-k}\oo\Sym \eta \to \L^{-k-1}\oo\Sym\eta \to \mc{B}\to 0.
\end{equation}
Consider the associated graded of these bundles. The argument above shows that for $k\gg 0$ the bundle $\gr \mc{B}$ contains no irreducible summands that yield $\ll$ in some cohomological degree. This implies that the $\ll$-isotypical component of $H^i(G/P,\, \mc{B})$ is zero, for all $i\geq 0$. The conclusion follows now by taking the long exact sequence in cohomology associated to (\ref{eq:bundlim}).
\end{proof}

The following exact sequence of $G$-equivariant $S$-modules
\[0\to S(-4) \xrightarrow{\;\;f\;\;} S \lra S/(f) \to 0\]
gives a long exact sequence in local cohomology modules
\begin{equation}\label{eq:long}
\dots \to H^{i-1}_{\ol{O}_2}(S/(f)) \to H^i_{\ol{O}_2}(S)(-4) \xrightarrow{\;f\;} H^i_{\ol{O}_2}(S) \to H^i_{\ol{O}_2}(S/(f)) \to H^{i+1}_{\ol{O}_2}(S)(-4) \to \dots
\end{equation}

To finish part (2) of Theorem \ref{thm:locring}, we pursue a reasoning similar to the one in \cite[Section 5]{senary}.

\subsection{Local cohomology supported in $\ol{O}_2$ for $(C_3,\omega_3)$} \label{sec:C3O2}

Throughout this section $(G',X)=(C_3,\omega_3)$.

\begin{lemma}\label{lem:grC3}
For all $k\gg 0$, the space of $G'$-invariants
\[H^{i} (G/P, \L^{-k}\oo\Sym (\gr \eta))^{G'}\] 
is nonzero if and only if $i=0,1,5,6$. Among these for $i=5,6$, the invariant spaces of $\bb{C}^*$-degree $-6$ are zero, and the space of $\bb{C}^*$-degree $-10$ is one-dimensional for $i=5$ and zero for $i=6$.
\end{lemma}

\begin{proof}
We start with the decomposition (\ref{eq:decogr}).
We apply Theorem \ref{thm:bott} to a summand of type $\V(2a,2b,-x)$ with $a,b,x\geq 0$, to see when it gives the trivial $G'$-representation $(0,0,0)$. Clearly, $(2a+1,2b+1,-x+1)$ is $(1,1,1)$ for $a=b=x=0$, giving the trivial representation for the $i=0$ cohomology. Otherwise, in order to continue we must have $x>1$ and reflecting at the third node gives $\V(2a+1,2b-2x+3, x-1)$. This gives $(1,1,1)$ for $a=0,b=1,x=2$. Otherwise, we must have $2x-2b+3<0$ and we proceed likewise to $i=2$. At the next step $i=3$, we can potentially reflect at either the first or third node. All in all, we can encounter the following weights for $i=2,3,3,4$ respectively:
\begin{equation}\label{eq:parity}
\begin{split}
(2a+2b+4-2x,2x-2b-3,2b+2-x), \quad (2x-2a-2b-4,2a+1,2b+2-x) , \\
 (2a+2b+4-2x,2b+1,x-2b-2), \quad (2x-2a-2b-4,2a+4b+5-2x,x-2b-2).
\end{split}
\end{equation}

Due to parity reasons, we can not get the trivial representation for $i=2,3,4$. After $i=5$ steps, we arrive at the weight (with $x\geq a+2b+3$):
\[(2b+1,2x-2a-4b-5, 2a+2b+3-x).\]
This gives $(1,1,1)$ for $a=1,b=0,x=4$. Otherwise, we must have $x>2a+2b+3$ and reflecting at the third node yields $(2b+1,2a+1,x-2a-2b-3)$. This gives $(1,1,1)$ for $a=0,b=0,x=4$ at $i=6$, and we stop since all entries of the weight vector are non-negative.

In other words, the only possible summands giving the trivial $G'$-representation $(0,0,0)$ are
\[\V(0,0,0), \, \V(0,2,-2), \, \V(2,0,-4), \, \V(0,0,-4),\]
when $i=0,1,5,6$, respectively. Now when the degree is $d-3k=-6$, we see by inspection that there are no summands $\V(2,0,-4)$ or $\V(0,0,-4)$ in (\ref{eq:decogr}). If $d-3k=-10$ then $\V(0,0,-4)$ is not a summand, but $\V(2,0,-4)$ is a summand in (\ref{eq:decogr}) with $a=1,b=0,c=k-3$, for $k\geq 4$.
\end{proof}

\begin{lemma}\label{lem:H4O2}
Each $G'$-module in the decomposition of $H^4_{\ol{O}_2}(S)$ is of the form $(2p+1)\omega_1 + 2q \omega_2 + r \omega_3$, for some $p,q,r \in \bb{Z}_{\geq 0}$. In particular, $H^4_{\ol{O}_2}(S)$ has no non-zero $G'$-invariant sections.
\end{lemma}

\begin{proof}
By Proposition \ref{prop:localopen}, any $G$-representation in $H^3_{\ol{O}_2}(S/(f))$ must appear in $H^{2} (G/P, \L^{-k}\oo\Sym (\gr \eta))$ for $k\gg 0$. We see from (\ref{eq:parity}) with $i=2$ that the only possible $G'$-representations in $H^{2} (G/P, \L^{-k}\oo\Sym (\gr \eta))$ are of the form $(2p+1)\omega_1 + 2q \omega_2 + r \omega_3$, for some $p,q,r \in \bb{Z}_{\geq 0}$.
The long exact sequence (\ref{eq:long}) gives an exact sequence
\[0 \to H^3_{\ol{O}_2}(S/(f)) \to H^4_{\ol{O}_2}(S)(-4) \xrightarrow{\;f\;} H^4_{\ol{O}_2}(S).\]
Since any element of $H^4_{\ol{O}_2}(S)$ is annihilated by a power of the $G'$-invariant $f$, this proves the claim.
\end{proof}

We now finish the proof of Theorem \ref{thm:locring}(b)(2).

\begin{prop}\label{prop:C3O2}
We have $H^i_{\ol{O}_2}(S) = 0$ for $i>4, i\neq 7$ and $H^7_{\ol{O}_2}(S)=E$.
\end{prop}

\begin{proof}
The modules $H^i_{\ol{O}_2}(S) = 0$ for $i>4$ have support contained in $\ol{O}_1$. Hence, by the description of the category in Theorem \ref{thm:quiv} (b), they must be direct sums of modules of type $L_1,E, M, N$, where $M=\D f^{-7/2}/\D f^{-3/2}$ and $N=\bb{D} N$ is the holonomic dual to $N$. There is a nonsplit exact sequence
\[0\to E \to N \to L_1\to 0.\]
The modules $L_1, M$ (resp. $E,N$) have $G$-semi-invariant elements of degree $-10$ (resp. $-14$) that are annihilated by $f$ (see Theorem \ref{thm:simples}). Since $E$ does not have a semi-invariant of degree $-6$, we see that $N$ also has a semi-invariant element of degree $-10$ annihilated by $f$. In particular, any non-trivial local cohomology module $H^i_{\ol{O}_2}(S)$ (for $i>4$) has a non-zero $G'$-invariant element in degree $-10$ or $-14$ annihilated by $f$.


Thus, by taking $G'$-invariants in degrees $-6$ and $-10$ in the long exact sequence (\ref{eq:long}), we see by Proposition \ref{prop:localopen}, Lemma \ref{lem:lim} and Lemma \ref{lem:grC3} that the only non-zero local cohomology (when $i>4$) is $H^7_{\ol{O}_2}(S)=E$. 
%
\end{proof}

We now proceed determining the $G$-character of the simple equivariant $\D$-module $L_2$.

\begin{theorem}\label{thm:D2}
When $(G',X)=(C_3,\omega_3)$, the $G$-character of $L_2$ is given by
\[[L_2]=\frac{t^{-7} \omega_1}{(1-t^{-2}\lie')(1-t^{-1}X)(1-X_4)(1-tX)(1-t^2\lie')}.\]
\end{theorem}

\begin{proof}
By Proposition \ref{prop:C3O2}, the long exact sequence (\ref{eq:long}) gives the exact sequence
\begin{equation}\label{eq:D2exact}
0\to H^3_{\ol{O}_2}(S/(f)) \to L_2(-4) \xrightarrow{\;f\;} L_2 \to H^4_{\ol{O}_2}(S/(f)) \to 0.
\end{equation}
For any representation $\ll$ of $G'$, this gives the formula (we use the notation as in Section \ref{sec:rep})
\begin{equation}\label{eq:recursive}
m_\ll^{d-4}(L_2) - m_\ll^d(L_2) =m^d_\ll(H^3(S/(f))) - m^d_\ll(H^4(S/(f))), \, \mbox{ for any } d\in \bb{Z}.
\end{equation}
By Lemma \ref{lem:H4O2}, we can assume that $\ll=(2p+1)\omega_1 + 2q \omega_2 + r \omega_3$, for some $p,q,r \in \bb{Z}_{\geq 0}$. We now proceed to compute the right-hand side of (\ref{eq:recursive}). 

For simplicity, put $N_i= H^{i} (G/P, \L^{-k}\oo\Sym (\gr \eta))$ ($k\gg 0$) and consider again the decomposition (\ref{eq:decogr}). Due to parity reasons, the only cohomology groups that yield a representation of the form $\ll$ via the Borel--Weil--Bott theorem occur in steps $i=2,3,4$ for weights of the form (\ref{eq:parity}). More precisely, the calculation shows that the representation $\ll=(2p+1)\omega_1 + 2q \omega_2 + r \omega_3$ occurs in $N_i$ in $\C^*$-degree $d \in \bb{Z}$ if and only if $d+r$ is odd and we are in the following situation:
\begin{equation}\label{eq:degrees}
\begin{cases}
d\leq 2p+r-3, & \mbox{ when } i=2;\\
d\leq 2p-r-5 \mbox{ or } d\leq -2p+r-7, & \mbox{ when } i=3;\\
d\leq -2p-r-9, & \mbox{ when } i=4. 
\end{cases}
\end{equation}
Moreover, $m^d_\ll(N_i)$ is equal to the number of the corresponding inequalities satisfied above. Since $H_{\ol{O}_2}^5(S/(f))=0$, by Proposition \ref{prop:localopen} and Lemma \ref{lem:lim} the representations in $N_4$ must cancel out completely via representations in $N_3$ in the spectral sequence corresponding to (\ref{eq:greta}). Hence, we have
\[m^d_\ll(H^3(S/(f))) - m^d_\ll(H^4(S/(f)))=m^d_\ll(N_2) - m^d_\ll(N_3)+m^d_\ll(N_4).\]
Together with (\ref{eq:recursive}) and (\ref{eq:degrees}), this yields the recursive formula in $d$ (when $d+r$ is odd):
\begin{equation} \label{eq:recform}
m_\ll^{d-4}(L_2) - m_\ll^d(L_2)= \begin{cases}
0 & \mbox{ when }  d>2p+r-3;\\
1 & \mbox{ when }  d\leq 2p+r-3 \mbox{ and } d>\max\{2p-r-5,-2p+r-7\};\\
0 & \mbox{ when } d\leq \max\{2p-r-5,-2p+r-7\} \mbox{ and } d>\min\{2p-r-5,-2p+r-7\};\\
\!\!\!-1 & \mbox{ when }  d\leq \min\{2p-r-5,-2p+r-7\} \mbox{ and } d>-2p-r-9;\\
0 & \mbox{ when } d\leq -2p-r-9.
\end{cases}
\end{equation}
We now show the initial condition $m_\ll^d(L_2)=0$, whenever $d>2p+r-7$. Assume by contradiction that this is not the case, hence there exists a non-zero highest weight vector $v\in L_2$ of weight $\ll$ with $\deg v>2p+r-7$. Since the support of $L_2$ is $\ol{O}_2$, there exist a minimal integer $l\geq 1$ with $f^l \cdot v =0$. Then the element $w=f^{l-1} \cdot v \neq 0$ has highest weight $\ll$ with $\deg w > 2p+r-7$. By the sequence (\ref{eq:D2exact}), this gives a non-zero element in $H^3_{\ol{O}_2}(S/(f))$ of highest weight $\ll$ with degree $>2p+r-3$. By Proposition \ref{prop:localopen}, this contradicts (\ref{eq:degrees}), showing that $m_\ll^d(L_2)=0$, whenever $d>2p+r-7$. This initial condition together with the recursive formula (\ref{eq:recform}) determines the character of $L_2$. To see that this matches the rational expression claimed in the statement of the theorem (denoted $[M]=\oplus_{d\in \bb{Z}} \,t^d M_d$), it is enough to show that the latter satisfies the same initial condition and recursive formula (\ref{eq:recform}).

It is easy to see that if $\ll$ is a representation that appears in some $M_d$, then it must be of the form $\ll=(2p+1)\omega_1 + 2q \omega_2 + r \omega_3$, for some $p,q,r \in \bb{Z}_{\geq 0}$. By inspection, the $\bb{Z}[t,t^{-1}]$-coefficient of such $\ll$ in $[M]$ is
\[c(t)=t^{-7}\cdot (t^{2p}+t^{2p-4}+ \dots + t^{4-2p} + t^{-2p})\cdot(t^r + t^{r-2}+ \dots + t^{2-r} + t^{-r}).\]
As the highest power of $t$ in $c(t)$ with non-zero coefficient is $2p+r-7$, the initial condition is satisfied. We are left to show that the coefficients of $c(t)$ satisfy the recursive formula (\ref{eq:recform}). We have
\[t^{4}\cdot c(t) - c(t) =(t^{2p-3}-t^{-2p-7})(t^r + t^{r-2}+ \dots + t^{2-r} + t^{-r})=t^{2p+r-3}+\dots +t^a - t^b -\dots - t^{-2p-r-7},\]
where $a=\max\{2p-r-3,-2p+r-5\}$ and $b=\min\{2p-r-5,-2p+r-7\}$, which agrees with (\ref{eq:recform}).
\end{proof}

\begin{remark}\label{rem:present}
We can give an explicit $\D$-module presentation for $L_2$ as follows. From Theorem \ref{thm:D2} we see that for $\ll=\omega_1$ we have $m^{-7}_\ll(L_2)=1$, and from Section \ref{sec:character} that $m^{-7}_\ll(M)=0$ for any other simple equivariant $\D$-module $M$. Denote by $V$ the irreducible $G$-representation corresponding to $\ll$ with $\C^*$-degree $-7$. Using (\ref{eq:PVHom}), we see that the $\D$-module $\mc{P}(V)$ is the projective cover of $L_2$ in $\opmod_G(\D_X)$. By Theorem \ref{thm:quiv}, in fact we have $\mc{P}(V)\cong L_2$. Now $\mc{P}(V)$ can be given an explicit presentation as explained in \cite[p. 435]{catdmod}.
\end{remark}

\subsection{Local cohomology supported in $\ol{O}_2$ for the other cases} \label{sec:otherO2}

In this section $(G',X)$ is one of the cases $(A_5,\omega_3), (D_6,\omega_5)$ or $(E_7,\omega_6)$. We adhere to uniformity as much as possible (for the case $(A_5,\omega_3)$, see \cite{senary}).

\begin{lemma}\label{lem:grother}
Let $(G',X)$ be either . Then for all $k\gg 0$, the space of $G'$-invariants
\[H^{i} (G/P, \L^{-k}\oo\Sym (\gr \eta))^{G'}\] 
is nonzero if and only if $i=0,1,m+1,m+2,2m+1,2m+2,3m+2,3m+3$. Among these, the spaces of $\bb{C}^*$-degree $-4m-2$ are one-dimensional for $i=0,1,m+1,m+2,2m+1$ and zero otherwise, and of $\bb{C}^*$-degree $-6m-4$ are one-dimensional for $i=0,1,m+1,m+2,2m+1,2m+2,3m+2$ and zero for $i=3m+3$.
\end{lemma}

\begin{proof}
First, we consider the case $(G',X)=(D_6,\omega_5)$. We apply Theorem \ref{thm:bott} to a summand of type $\V(0,a,0,b,-x,0)$ with $a,b,x\geq 0$ in (\ref{eq:decogr}). Computing as in Lemma \ref{lem:grC3} we see that we obtain the trivial $G'$-representation only for the following types:
\[\V(0,0,0,0,0,0), \, \V(0,0,0,1,-2,0), \, \V(0,2,0,1,-6,0), \, \V(0,3,0,0,-6,0),\]
\[\V(0,0,0,3,-10,0),\,\V(0,1,0,2,-10,0),\,\V(0,1,0,0,-10,0),\,\V(0,0,0,0,-10,0),\]
when $i=0,1,5,6,9,10,14,15$, respectively. Assuming $k\gg 0$, we see by inspection that the corresponding summands appear in the decomposition (\ref{eq:decogr}) for degree $d-3k=-18$ only when $i=0,1,5,6,9$, and for degree $d-3k=-28$ all but the last summand $\V(0,0,0,0,-10,0)$ appear.

Next, consider the case $(G',X)=(E_7,\omega_6)$. We apply Theorem \ref{thm:bott} to a summand of type $\V(a,0,0,0,b,-x,0)$ with $a,b,x\geq 0$ in (\ref{eq:decogr}). Computing as above we see that we obtain the trivial $G'$-representation only for the following types:
\[\V(0,0,0,0,0,0,0), \, \V(0,0,0,0,1,-2,0), \, \V(4,0,0,0,1,-10,0), \, \V(5,0,0,0,0,-10,0),\]
\[\V(0,0,0,0,5,-18,0),\,\V(1,0,0,0,4,-18,0),\,\V(1,0,0,0,0,-18,0),\,\V(0,0,0,0,0,-18,0),\]
when $i=0,1,9,10,17,18,26,27$, respectively. The rest of the proof follows similarly to the previous case.
\end{proof}

The following completes the proof of Theorem \ref{thm:locring}(a)(2) (cf. \cite[Section 5.1]{senary}).

\begin{prop}\label{prop:O2other}
The only non-zero modules $H^i_{\ol{O}_2}(S)$ for $i>m+3$ are  $H^{2m+3}_{\ol{O}_2}(S)=L_1, \, H^{3m+4}_{\ol{O}_2}(S)=E$.
\end{prop}

\begin{proof}
The modules $H^i_{\ol{O}_2}(S) = 0$ for $i>m+3$ have support contained in $\ol{O}_1$. Hence, by the description of the category in Theorem \ref{thm:quiv} (a), they must be direct sums of modules of type $L_1$ and $E$.

The modules $L_1$ (resp. $E$) have $G$-semi-invariant elements of degree $-4m-6$ (resp. $-6m-8$) that are annihilated by $f$ (see Theorem \ref{thm:simples}). Moreover, multiplication by $f$ is surjective both on $L_1^{G'}$ and on $E^{G'}$ (see Theorem \ref{thm:simples}) and also on the $G'$-invariant space of $H^{m+3}_{\ol{O}_2}(S) \cong S_f \sqrt{f}/\D f^{r_1+1}$. By the long exact sequence (\ref{eq:long}) together with Proposition \ref{prop:localopen} and Lemma \ref{lem:grother}, we conclude that $L_1$ (resp. $E$) can only appear in the local cohomology modules $H^i_{\ol{O}_2}(S)$ for $i=m+4,2m+3$ (resp. $i=m+4,2m+3,2m+4,3m+4$). Since there are no cancelations possible for $i=2m+3$ (resp. $i=3m+4$) in the spectral sequence when passing from the associated graded $\gr \eta$ to $\eta$ (\ref{eq:greta}), we see that $L_1$ is a summand of $H^{2m+3}_{\ol{O}_2}(S)$ (resp. $H^{3m+4}_{\ol{O}_2}(S)=E$). We will now show that the rest of the terms must cancel out in the spectral sequence.

We observe that $H^{m+3}_{\ol{O}_2}(S)= S_f \sqrt{f}/\D f^{r_1+1}$ has no $G$-semi-invariant of degree $-4m-2$ (resp. $-6m-4$) that is annihilated by $f$. Together with the sequence (\ref{eq:long}), Proposition \ref{prop:localopen} and Lemma \ref{lem:lim}, this implies that for $i=m+1,m+2$ the semi-invariants in degree $-4m-2$ (resp. $-6m-4$) from Lemma \ref{lem:grother} must cancel each other out in the spectral sequence.

We are left to show that the semi-invariants in Lemma \ref{lem:grother} in degree $-6m-4$ between $i=2m+1,2m+2$ cancel each other out in the spectral sequence. This can be seen as in the proof of \cite[Proposition 5.9]{senary}, by comparing the relevant connecting homomorphism between them to another one in a different degree. We sketch the argument as follows. 
The sequence (\ref{eq:eta}) induces the following short exact sequences:
\[\begin{array}{lcccr}
0 \to \,\, \L^{-k}\oo \Sym_d \Omega_{G/P}\oo \O(d) & \longrightarrow & \L^{-k} \oo \Sym_d \eta & \longrightarrow & \L^{-k} \oo \Sym_{d-1}\eta \oo \O(1) \,\,\to 0 \\
0 \to \,\, \L^{-k}\oo \Sym_{d-1} \Omega_{G/P} \oo \O(d) & \longrightarrow & \L^{-k} \oo \Sym_{d-1} \eta \oo \O(1) &\longrightarrow & \L^{-k} \oo \Sym_{d-2}\eta \oo\O(2) \,\,\to 0 \\
 \qquad \qquad \vdots & & \vdots && \vdots \qquad\qquad\qquad\qquad\\
 0 \to \,\, \L^{-k}\oo \Omega_{G/P}\oo\O(d) \quad & \longrightarrow & \L^{-k} \oo \eta \oo \O(d-1) & \longrightarrow & \L^{-k} \oo \O(d) \quad \,\,\to 0
\end{array}\]
From the proof of Lemma \ref{lem:grother}, we need to track the two irreducible bundles in $\L^{-k} \oo \Sym (\gr \eta)$ yielding a $G'$-invariant in degree $-6m-4$ in cohomological degrees $i=2m+1$ and $2m+2$. Putting $d=3k-6m-4$ with $k\gg 0$, we get that the only potentially non-zero connecting homomorphism on the level of $G'$-invariants (between $1$-dimensional spaces) is induced by the $(m+1)$th exact sequence from above (recall that $\L \cong \O(1)$):
\begin{equation}\label{eq:connect}
\begin{aligned}
0 \to H^{2m+1}(G/P, \, \Sym_{d-m}\eta \oo \O(m-k))^{G'} \to H^{2m+1}(G/P, \, \Sym_{d-m-1}\eta \oo \O(m+1-k))^{G'}\longrightarrow \\
\longrightarrow H^{2m+2}(G/P, \, \Sym_{d-m} \Omega_{G/P} \oo \O(d-k))^{G'} \to H^{2m+2}(G/P, \, \Sym_{d-m}\eta \oo \O(m-k))^{G'} \to 0.
\end{aligned}
\end{equation}

Now put $k\to k-2$ and $d \to d-2$. Then repeating the discussion above gives a $G'$-invariant in degree $d-3k = -6m$ and in cohomological degrees $2m+1$ and $2m+2$ again. The relevant connecting homomorphism comes from the $(m-1)$th exact sequence from above, and it is precisely the same as the connecting homomorphism in (\ref{eq:connect}) (ignoring the $\C^*$-action). Since $E^{G'}$ and $L_1^{G'}$ have no elements in degrees $-6m$ and $-6m-4$, we see from (\ref{eq:long}) that $H^{2m+2}_{\ol{O}_2}(S/(f))^{G'}$ and $H^{2m+3}_{\ol{O}_2}(S/(f))^{G'}$ have no elements in degree $-6m$. By Proposition \ref{prop:localopen} and Lemma \ref{lem:lim}, this implies that the connecting homomorphism (\ref{eq:connect}) must be non-trivial.
\end{proof}

\section{Other invariants}\label{sec:other}

\subsection{Lyubeznik numbers and intersection cohomology groups of orbit closures}\label{sec:lyubint}

In this section, we determine some local cohomology groups of equivariant $\D$-modules with support in the origin. We then use these computations to determine the Lyubeznik numbers (for the $(A_5,\omega_3)$ case see \cite[Section 5]{senary}) and the (middle perversity) intersection cohomology groups of the orbit closures $\ol{O}_1,\ol{O}_2,\ol{O}_3$.

We start with an observation relating the intersection cohomology groups to the local cohomology groups.

\begin{prop}\label{prop:IH}
Let $p\in\{1,2,3\}$ and $c_p= \codim_X O_p$. For all $i\in \bb{Z}$, we have
\[\mbox{\large\(H_{\{0\}}^{i+c_p}(\F(L_p))\)} = \mbox{\large\(E\)}^{\mbox{\small\(\oplus \dim IH^i(\ol{O}_p)\)}}. \]
\end{prop}

\begin{proof}
By the Riemann--Hilbert correspondence (especially \cite[Theorem 7.1.1]{htt}), the intersection cohomology groups $IH^i(\ol{O}_p)$ can be computed as the (derived) pushforward of the module $L_p$ to a point. The latter is equivalent to the restriction of the Fourier transform $\F(L_p)$ to a point \cite[Proposition 3.2.6]{htt}. By \cite[Proposition 1.7.1]{htt}, this can be computed as local cohomology supported at the origin.
\end{proof}

\begin{lemma}\label{lem:locD1}
The module $H^i_{\{0\}}(L_1)$ is non-zero (in which case it is isomorphic to $E$) if and only if:
\begin{itemize}
\item[(a)] When $(G',X)$ is $(A_5,\omega_3),\, (D_6, \omega_5)$ or $(E_7,\omega_6): \, i=m+2,2m+2,3m+4$;
\item[(b)] When $(G',X)=(C_3,\omega_3): \, i=1,\, 3m+4$.
 \end{itemize}
\end{lemma}

\begin{proof}
This follows readily by considering the spectral sequence $H^i_{\{0\}}(H^j_{\ol{O}_1}(S)) \Rightarrow H^{i+j}_{\{0\}}(S)$ together with Theorem \ref{thm:locring}. We give more details for part (b). In this case, the latter spectral sequence degenerates already on the second page, so that we have $H^{3m+4}_{\{0\}}(H^{3m+4}_{\ol{O}_1}(S)) = E$, and $H^{i}_{\{0\}}(H^{3m+4}_{\ol{O}_1}(S)) =0$ for $i\neq 3m+4$. The claim now follows by using the long the exact sequence associated to $0\to L_1 \to H^{3m+4}_{\ol{O}_1}(S) \to E \to 0$.
\end{proof}

\begin{lemma}\label{lem:locD2}
The module $H^i_{\{0\}}(L_2)$ is non-zero (in which case it is isomorphic to $E$) if and only if:
\begin{itemize}
\item[(a)] When $(G',X)$ is $(A_5,\omega_3),\, (D_6, \omega_5)$ or $(E_7,\omega_6)$: \, $i=m+3,2m+3,3m+3,3m+5,4m+5,5m+5$;
\item[(b)] When $(G',X)=(C_3,\omega_3)$: \, $i=m+3,\, 5m+5$.
 \end{itemize}
\end{lemma}

\begin{proof}
We consider the spectral sequence $H^i_{\{0\}}(H^j_{\ol{O}_2}(S)) \Rightarrow H^{i+j}_{\{0\}}(S)$ together with Theorem \ref{thm:locring}. Part (b) follows easily. For part (a) assuming that $H^{i}_{\{0\}}(H^{m+3}_{\ol{O}_2}(S)) =0 $ for $i<3m+3$, it follows from the spectral sequence (using Lemma \ref{lem:locD1}) that
\[H^{i}_{\{0\}}(H^{m+3}_{\ol{O}_2}(S)) =E, \mbox{ for } i=3m+3,4m+5,5m+5,\]
and it is zero for all other $i$. Using the exact sequence $0\to L_2 \to H^{m+3}_{\ol{O}_2}(S) \to L_1\to 0$ proves the claim.

We are left to show the vanishing of $H^{i}_{\{0\}}(H^{m+3}_{\ol{O}_2}(S))$ for $i<3m+3$. Note that $H^{m+3}_{\ol{O}_2}(S))\cong H^1_{\ol{O}_3}(L_4')$, and also $H^i_{\{0\}}(H^1_{\ol{O}_3}(L_4')) = H^{i+1}_{\{0\}}(L_4')$, since the corresponding spectral sequence degenerates. Hence, we are left to show that $H^{i}_{\{0\}}(L_4')=0$, for $i<3m+4$. Since $\F(L_4')=L_1$, by Proposition \ref{prop:IH} the latter follows since $\codim_X O_p=3m+4$.
\end{proof}

\begin{lemma}\label{lem:locD3}
The module $H^i_{\{0\}}(L_3)$ is non-zero (in which case it is isomorphic to $E$) if and only if:
\begin{itemize}
\item[(a)] When $(G',X)$ is $(A_5,\omega_3),\, (D_6, \omega_5)$ or $(E_7,\omega_6)$: \, $i=1,6m+7$;
\item[(b)] When $(G',X)=(C_3,\omega_3)$: \, $i=3m+4,\, 6m+7$.
 \end{itemize}
\end{lemma}

\begin{proof}
For part (a) we have $H^i_{\{0\}}(H^1_{\ol{O}_3}(S)) = H^{i+1}_{\{0\}}(S)$, since the corresponding spectral sequence degenerates. Hence, the claim follows readily by Theorem \ref{thm:locring}. 

For part (b), since $\F(L_3)=L_1$ it is enough to compute $IH^i(\ol{O}_1)$ by Proposition \ref{prop:IH}. Because $\ol{O}_1$
is the affine cone over $G/P$, the latter can be computed as explained in \cite[Section I.5]{borel}. Using the Poincar\'e polynomial (\ref{eq:poin}) we get that $\dim IH^i(\ol{O}_1)=1$ for $i=0,6$ and it is $0$ otherwise.
\end{proof}

Our results above determine also the Lyubeznik numbers $\ll_{i,j}(R_p)$ (see \cite{lyubeznik}) of the orbit closures $\ol{O}_p$ for $p=1,2,3$, where $R_p:=\C[\ol{O}_p]_{\mf{m}}$ is the localization of the coordinate ring of $\ol{O}_p$ at the maximal homogeneous ideal $\mf{m}$. Since $\ll_{i,j}(R_p)$ equals the multiplicity of $E$ in $H^i_{\{0\}}(H^{\dim X-j}_{\ol{O}_p}(S))$, we obtain the following by our previous calculations in this section together with Theorem \ref{thm:locring}.

\begin{corollary}\label{cor:lyubez}
For $p\in\{1,2,3\}$ the following are the only non-zero Lyubeznik numbers $\ll_{i,j}(R_p)$ (in which case they are equal to 1):
\begin{itemize}
\item[(a)] If $(G',X)$ is $(A_5,\omega_3),\, (D_6, \omega_5)$ or $(E_7,\omega_6)$:
\begin{itemize}
\item[(1)] When $p=1: \, (i,j)=(0, m+3), \, (0, 2m+3), \, (m+2, 3m+4),\, (2m+2, 3m+4),\, (3m+4, 3m+4);$
\item[(2)] When $p=2: \, (i,j)=(0, 3m+4), \, (m+2, 4m+5), \, (2m+2, 4m+5), \, (3m+4, 4m+5),\\ 
{} \hspace{1.45in} (3m+3, 5m+5), \, (4m+5, \,5m+5), \, (5m+5, 5m+5);$
\item[(3)] When $p=3: \,(i,j)=(6m+7,6m+7).$
\end{itemize}
\item[(b)] If $(G',X)$ is $(C_3,\omega_3)$:
\begin{itemize}
\item[(1)] When $p=1: \, (i,j)=(3m+4,3m+4);$
\item[(2)] When $p=2: \, (i,j)= (0,3m+4),\, (2m+2, 5m+5), \, (5m+5, 5m+5) ;$
\item[(3)] When $p=3: \, (i,j)=(6m+7,6m+7).$
\end{itemize}
\end{itemize}
\end{corollary}

Finally, we give a list of the intersection cohomology groups of the orbit closures. These follow by our previous calculations together with Proposition \ref{prop:IH}.

\begin{corollary}\label{cor:ic}
For $p\in\{1,2,3\}$ the following are the only non-zero intersection cohomology groups $IH^i(\ol{O}_p)$ (in which case they are 1-dimensional):
\begin{itemize}
\item[(a)] If $(G',X)$ is $(A_5,\omega_3),\, (D_6, \omega_5)$ or $(E_7,\omega_6)$:
\begin{itemize}
\item[(1)] When $p=1$, then $i=0, \, m+2, \, 2m+2;$
\item[(2)] When $p=2$, then $i=  0, \, m, \, 2m, \, 2m+2, \, 3m+2, \, 4m+2 ;$
\item[(3)] When $p=3$, then $i=  0, \, 6m+6 .$
\end{itemize}
\item[(b)] If $(G',X)$ is $(C_3,\omega_3)$:
\begin{itemize}
\item[•] When $p=1, 2, 3$, then $i=0, \, 4m+2.$
\end{itemize}
\end{itemize}
\end{corollary}

\subsection{Gorenstein property and Castelnuovo--Mumford regularity}\label{sec:last}

In the last section, we establish some results for Gorenstein varieties based on observations that we extracted from the previous sections.

Throughout this section $G=G'\times \C^*$ denotes a linearly reductive complex connected algebraic group, and $X$ is a finite-dimensional rational representation of $G$. Here the factor $\C^*$ of $G$ acts on $X$ by the usual scaling.

We start with a result describing the elements of a local cohomology module that are annihilated by its supporting ideal, which is relevant even in the case when $G'=\{1\}$.

\begin{lemma}\label{lem:gor}
Let $Z$ be a Gorenstein $G$-stable closed subvariety of $X$ with $c=\codim_X Z \,\, >0$, and $I$ the defining ideal of $Z$. Then there is a $G$-equivariant isomorphism of $S$-modules
\[\Hom_S (S/I, \, H^c_Z(S)\,) \,\, \cong \,\, S/I \oo \chi \,,\] 
for some character $\chi: G \to \C^*$. In particular, the module $\,H^c_Z\,(S)\,$ has a unique (up to scalar) $G$-semi-invariant section $h$ of degree $\, - \CM(I) - c$ with $\Ann_S(h) = I$ (here $\CM(I)$ is the Castelnuovo--Mumford regularity of $I$).
\end{lemma}

\begin{proof}
We follow the proof of \cite[Proposition 3.1 (1)]{hunkoh} by additionally keeping track of equivariance. First, $\Hom_S (S/I, \, M) \cong  \Hom_S (S/I, \, H^0_Z(M))$, for any $S$-module $M$. Hence, we have a spectral sequence
\[\Ext^i_S(S/I, \, H^j_Z(S) ) \, \Longrightarrow \, \Ext_S^{i+j}(S/I, S).\]
This yields a $G$-equivariant isomorphism
\[\Hom_S (S/I, \, H^c_Z(S)\,) \, \cong \, \Ext_S^c(S/I,S).\]
Since $Z$ is Gorenstein, we have a $G$-equivariant isomorphism $\Ext_S^c(S/I,S) \cong S/I \oo \chi$ for a character $\chi$.
\end{proof}

The result above provides an interesting technique for proving that a variety is not Gorenstein. For example, by Lemma \ref{lem:H4O2} we obtain that for $(G',X)=(C_3,\omega_3)$ the variety $\ol{O}_2$ is not Gorenstein. Similarly, \cite[Lemma 3.4]{bindmod} implies the (well-known) result that the affine cone over the twisted cubic curve is not Gorenstein.

Since the $\D_X$-module $H^c_Z(S)$ has a unique simple submodule $L$ (corresponding to the intersection cohomology sheaf of the trivial local system on $Z_{reg}$), it is interesting to see when the element $h$ as above lies in $L$. While this happens frequently (e.g., for our subexceptional series), it is not always the case as can be seen already when $Z$ is a hypersurface (e.g., the discriminant of cubics \cite{bindmod}).

In \cite[Conjecture 5.17]{leva}, a conjecture has been made for the existence of (semi)-invariant sections for some $\D$-modules. The conjecture has been disproved precisely for orbit closures that are not Gorenstein (see \cite{claudiu5} and \cite[Proposition 5.8]{catdmod}). On the other hand, from \cite{kraswey}, \cite{kraswey2} and \cite{kraswey3}, we see that indeed all orbit closures in our exceptional series $(G',X)$ are Gorenstein, with the only exception for $\ol{O}_2$ when $(G',X)=(C_3,\omega_3)$ due to the reason mentioned above (nevertheless, its regularity is obtained in \cite[Page 38]{kraswey}). In conclusion, when the group $G$ is large enough, the existence of a semi-invariant section as in Lemma \ref{lem:gor} gives strong evidence for the Gorenstein property of a $G$-stable subvariety $Z$.

\begin{prop}\label{prop:bscm}
Consider $Z$ as in Lemma \ref{lem:gor}. Assume $\C[X]^{G'} = \C[f]$, and that there is a surjective map of $\D$-modules  $\pi: \D f^\alpha \twoheadrightarrow \,H^c_Z\,(S)\,$, for some $\alpha \in \mathbb{Q}$. Let $r\in \, \alpha+\bb{N}$ be maximal with the property $\pi( f^{r} ) \neq 0$. 

Then $r$ is a root of the Bernstein-Sato polynomial of $f$ and $\,\CM(I) \,= \,- \, r \, \cdot \, \deg f \, - c$. 
\end{prop}

\begin{proof}
The existence of the map $\pi$ implies that we have $f\in I$. In particular, this shows the existence of $r \in\, \alpha+\bb{N}$. By \cite[Proposition 4.9]{catdmod}, $r$ is a root of the Bernstein-Sato polynomial of $f$. 

Since $\C[X]^{G'} = \C[f]$, the only semi-invariants (up to constant) in $\D f^\alpha$ are powers of $f$. Hence, the only semi-invariant in $H^c_Z\,(S)$ that is annihilated by $f$ is $\pi(f^r)$. By Lemma \ref{lem:gor}, $H^c_Z\,(S)$ has a semi-invariant section $h$ annihilated by $I$. Since $h$ is annihilated by $f$, this shows that $h=\pi(f^r)$ (up to non-zero constant).
\end{proof}

\begin{remark}\label{rem:piquiv}
The existence of a map $\pi$ can often be seen directly from the quiver of $\opmod_G(\D_X)$. Namely, the module $S_f \cdot f^\alpha$ (for $\alpha$ such that this is $G$-equivariant) is an injective object in $\opmod_G(\D_X)$, and the module $H^c_O(S)$ is an injective object in $\opmod_G^{\ol{O}}(\D_X)$ for an orbit $O$ (see \cite[Lemma 2.4]{bindmod}, \cite[Lemma 3.11]{catdmod}).
\end{remark}

We established a fundamental link between the roots of the Bernstein--Sato polynomial of $f$ and Castelnuovo-Mumford regularity of Gorenstein varieties that appear in the localizations at powers of $f$ as above. The assumption for the existence of $\pi$ as above is satisfied frequently for prehomogeneous vector spaces with semi-invariants. It is satisfied for our subexceptional series for all (Gorenstein) orbit closures. Using Remark \ref{rem:piquiv}, it is not difficult to see that in all the cases encountered in \cite{bindmod}, \cite{mike} and \cite[Section 5]{catdmod}, the map $\pi$ exists for Gorenstein orbit closures with only one exception from \cite[Section 5.5]{catdmod}:
\begin{example}\label{ex:spin9}
Consider $G'=\operatorname{Spin}(9)$ and let $X$ be its 16-dimensional spin representation. The group $G$ acts on $X$ with 4 orbits $O_0,\, O_1, \, O_2, \, O_3$ of codimensions $16, \, 5,\, 1, \, 0$, respectively. Then $\C[X]^{G'}=\C[f]$, with $\deg f = 2$.  The roots of the Bernstein-Sato polynomial of $f$ are $-1, \, -8$. It follows from \cite[Section 5.1]{kraswey} that $\ol{O}_1$ is Gorenstein with $\reg(\ol{O}_1) = 3$. Since $H^5_{\ol{O}_1}(S)$ is a simple $\D$-module that corresponds to an isolated vertex of the quiver of $\opmod_G(\D_X)$ (see \cite[Section 5.5]{catdmod}), it is not a composition factor of $S_f \cdot f^\alpha$, for any $\alpha \in \mathbb{Q}$. For $\ol{O}_0, \, \ol{O}_2$ the map $\pi$ exists, and we have $1=\reg(\ol{O}_2)=-(-1)\cdot 2 - 1$ and $0=\reg(O_0)=-(-8)\cdot 2 - 16$.

The exceptional behavior of $\ol{O}_1$ can be also seen from the fact that it is a self-dual highest weight orbit closure \cite{knopmenzel}. Moreover, both of the simple equivariant $\D$-modules corresponding to $O_2$ and $O_1$ (by Lemma \ref{lem:gor}) have $G'$-invariant elements in degree $-8$. This also demonstrates the sharpness of \cite[Corollary 3.23]{catdmod}, as $X$ is a spherical $G$-variety that is not of Capelli type.
\end{example}

Finally, we apply the results above to our subexceptional series again.

\begin{corollary}\label{cor:CM}
For $p\in\{1,2\}$ the Castelnuovo--Mumford regularity $\reg(I_p)$ of the defining ideal $I_p$ of $\ol{O}_p$ is:
\begin{itemize}
\item[(a)] If $(G',X)$ is $(A_5,\omega_3),\, (D_6, \omega_5)$ or $(E_7,\omega_6): \,\,\reg(I_1)=m+2 ,\,\, \reg(I_2)=m+3;$
\
\item[(b)] If $(G',X)$ is $(C_3,\omega_3): \,\, \reg(I_1)= \reg(I_2) = m+2.$
\end{itemize}
\end{corollary}

\section*{Acknowledgments}

We are grateful to Michael Perlman and Claudiu Raicu for helpful conversations and suggestions. Weyman was partially supported by NSF grant DMS 1802067, Sidney Professorial Fund and National Science Centre, Poland grant UMO-2018/29/BST1/01290.

\bibliographystyle{alpha}
\bibliography{biblo}

\end{document}